\documentclass{amsart}

\usepackage{amssymb}
\usepackage{amsmath}
\usepackage{xcolor}
\usepackage{xypic}

\newcommand{\brackets}[1]{
	 #1
}

\newcommand{\one}[2]{
	\brackets{#1} \ar @{-} [#2]
}

\newcommand{\two}[1]{
	\brackets{#1} \ar @{-} [lu] \ar @{-} [ru]
}

\newcommand{\three}[1]{
	\brackets{#1} \ar @{-} [u] \ar @{-} [rru]
}

\newcommand{\four}[1]{
	\brackets{#1} \ar @{-} [u] \ar @{-} [llu]
}

\newcommand{\five}[1]{
	\brackets{#1} \ar @{-} [u]
}

\def\tl{\triangleleft}
\def\AGL{{\rm AGL}}

\def\mod{{\rm mod\ }}

\def\fix{{\rm fix}}

\def\Aut{{\rm Aut}}
\def\Stab{{\rm Stab}}
\def\la{\langle}
\def\ra{\rangle}
\def\Z{{\mathbb Z}}
\def\soc{{\rm soc}}
\def\PSL{{\rm PSL}}
\def\PGammaL{{\rm P\Gamma L}}

\def\PSL{{\rm PSL}}

\def\PGL{{\rm PGL}}

\def\supp{{\rm supp}}

\def\Sym{{\rm Sym}}
\def\Alt{{\rm Alt}}
\def\Cay{{\rm Cay}}

\newtheorem{thm}{Theorem}[section]
\newtheorem*{thm*}{Theorem}
\newtheorem*{conj*}{Conjecture}
\newtheorem{cor}[thm]{Corollary}
\newtheorem{conj}[thm]{Conjecture}
\newtheorem{lem}[thm]{Lemma}

\theoremstyle{remark}
\newtheorem{rem}[thm]{Remark}

\theoremstyle{definition}
\newtheorem{eg}[thm]{Example}
\newtheorem{defn}[thm]{Definition}
\newtheorem{prob}[thm]{Problem}

\usepackage{hyperref}

\title{Cayley graphs of more than one abelian group}

\author{Ted Dobson}            
\address{IAM, University of Primorska\\
Muzejska trg 2\\
Koper 6000, Slovenia}    
\email{ted.dobson@upr.si}

\author{Joy Morris}            
\address{Department of Mathematics and Computer Science\\
University of Lethbridge\\
Lethbridge, AB. T1K 3M4}    
\email{joy.morris@uleth.ca}



\begin{document}
\maketitle

\begin{abstract}
We show that for certain integers $n$, the problem of whether or not a Cayley digraph $\Gamma$ of $\Z_n$ is also isomorphic to a Cayley digraph of some other abelian group $G$ of order $n$ reduces to the question of whether or not a natural subgroup of the full automorphism group contains more than one regular abelian group up to isomorphism (as opposed to the full automorphism group).  A necessary and sufficient condition is then given for such circulants to be isomorphic to Cayley digraphs of more than one abelian group, and an easy-to-check necessary condition is provided.
\end{abstract}

\section{Introduction}

It is well known that a Cayley digraph of a group $G$ may also be isomorphic to a Cayley digraph of group $H$ where $G$ and $H$ are not isomorphic.  A natural question is then to determine exactly when a Cayley digraph is isomorphic to a Cayley digraph of a nonisomorphic group.  Perhaps the first work on this problem was by Joseph in 1995 \cite{Joseph1995} where she determined necessary and sufficient conditions for a Cayley digraph of order $p^2$, $p$ a prime, to be isomorphic to a Cayley digraph of both groups of order $p^2$ (see \cite[Lemma 4]{DobsonW2002} for a group theoretic version of this result).  The second author \cite{Morris1999} subsequently extended this result and determined necessary and sufficient conditions for a Cayley digraph of $\Z_{p^k}$, $k\ge 1$ and $p$ an odd prime, to be isomorphic to a Cayley digraph of some other abelian group (see Theorem \ref{pktwogroups} for the statement of this result).  Additionally, she found necessary and sufficient conditions for a Cayley digraph of $\Z_{p^k}$ to be isomorphic to a Cayley digraph of any group of order $p^k$.  The equivalent problem for $p=2$ (when both groups are abelian) was solved by Kov\'acs and Servatius \cite{KovacsS2012}.  Digraphs of order $pq$ that are Cayley graphs of both groups of order $pq$, where $q\mid (p-1)$ and $p,q$ are distinct primes were determined by the first author in \cite[Theorem 3.4]{Dobson2006a}.  Finally, Maru\v si\v c and the second author studied the question of which normal circulant graphs of square-free order are also Cayley graphs of a nonabelian group \cite{MarusicM2005}.

We show in this paper that for some values of $n$, we can reduce the problem of which circulant digraphs of order $n$ are also Cayley digraphs of some other abelian group of order $n$, to the prime-power case previously solved by the second author.  Specifically, let $n = p_1^{a_1}p_2^{a_2}\dotsm p_r^{a_r}$ and let $k = p_1\dotsm p_r$, where each $p_i$ is prime. Then the reduction works if $\gcd(k,\varphi(k)) = 1$.

At first glance, this arithmetic condition may seem odd.  However, this condition is also contained in a well-known result of P\'alfy \cite{Palfy1987} where he characterized all finite groups which are CI-groups with respect to all classes of combinatorial objects.  His result is that the only such groups are groups of order $4$ and cyclic groups of order $n$ satisfying $\gcd(n,\varphi(n)) = 1$.  An equivalent statement of this theorem is that a group $G$ of order $n$ has the property that any subgroup $H\le S_n$ with a regular subgroup isomorphic to $G$ has one conjugacy class of regular subgroups isomorphic to $G$ if and only if $n = 4$ or $\gcd(n,\varphi(n)) = 1$.  See \cite{Dobson2003a,Dobson2014} for some generalizations of P\'alfy's Theorem.  In addition to completely answering the question of which groups are CI-groups with respect to every class of combinatorial objects, P\'alfy's Theorem has also been used to classify various classes of vertex-transitive graphs \cite{Dobson2000a,Dobson2008,DobsonM2011}, and using these results, the first author with Pablo Spiga \cite{DobsonS2017} showed that there are Cayley numbers with arbitrarily many prime divisors, settling an old problem of Praeger and McKay.  Thus P\'alfy's Theorem and its generalizations not only have the obvious applications to the isomorphism problem for Cayley objects, but also to classification problems.

Our approach is to consider the values of $n$ with the following property:  Any subgroup $H\le S_n$ that contains a regular subgroup isomorphic to $\Z_n$ and some other regular abelian group $G$ has a nilpotent subgroup which contains conjugates of every regular abelian subgroup of $H$.  We show in Theorem \ref{maingroup} that $n$ has this property if and only if $\gcd(k,\varphi(k)) = 1$ (with $k$ as defined above).  In a sense, our result shows that for these special values, the question of when a Cayley object of $\Z_n$ is also a Cayley object of some other abelian group reduces to the prime-power case as a finite nilpotent group is the direct product of its Sylow subgroups.  As our result is a characterization of such values of $n$, the full automorphism groups of classes of combinatorial objects may have the above property for all values of $n$, but the structure of the combinatorial object will have to be used to prove this - permutation group theoretic techniques will not suffice.

Next, we apply this result when the combinatorial object are digraphs - the only combinatorial objects for which the prime-power case has been solved.  We show in Theorem \ref{main} that if all automorphism groups of circulant digraphs have the permutation group theoretic property in the previous paragraph (which is the case if $\gcd(k,\varphi(k)) = 1$), then the question of when a circulant digraph is also a Cayley digraph of some other abelian groups reduces to the prime-power case.

Theorem \ref{maingroup} also turns out to also be a generalization of  a restricted form of P\'alfy's Theorem.  This ultimately follows as the only regular abelian subgroup of $\Z_n$ when $n$ is square-free is $\Z_n$, and a transitive nilpotent group of square-free order is necessarily $\Z_n$.  For a complete explanation, see Remark \ref{Palfy gen}. This gives that P\'alfy's Theorem and its generalizations also have applications to a third question, the question of when a combinatorial object is a Cayley object of more than one group.

To summarize, we determine an algebraic condition which reduces the problem of when a circulant digraph is also a Cayley digraph of some other abelian group to the prime-power case.  We determine for which values of $n$ this condition holds for {\it all} permutation groups of degree $n$. We then combine these results to determine necessary and sufficient conditions for a circulant digraph to also be a Cayley digraph of some other abelian group for those values of $n$.

In the remainder of this section, we state the second author's result (Theorem \ref{pktwogroups}), first providing the necessary definitions for that statement. In Section 2, we provide the necessary group theoretic results to prove our main theorem. In Section 3, we provide the necessary graph theoretic results to prove our main result, which is Corollary \ref{mainresult}.

\begin{defn}
Let $G$ be a group and $S\subset G$. Define a {\bf Cayley digraph of $G$}, denoted $\Cay(G,S)$, to be the digraph with $V(\Cay(G,S)) = G$ and $A(\Cay(G,S)) = \{(g,gs):g\in G, s\in S\}$.  We call $S$ the {\bf connection set} of $\Cay(G,S)$.
\end{defn}

\begin{defn}
Let $\Gamma_1$ and $\Gamma_2$ be digraphs.  The {\bf wreath product of $\Gamma_1$ and $\Gamma_2$}, denoted $\Gamma_1\wr\Gamma_2$, is the digraph with vertex set $V(\Gamma_1)\times V(\Gamma_2)$ and arcs $((u,v)(u,v'))$ for $u\in V(\Gamma_1)$ and $(v,v')\in A(\Gamma_2)$
or $((u,v),(u',v'))$ where $(u,u')\in A(\Gamma_1)$ and $v,v'\in V(\Gamma_2)$.
\end{defn}

For any terms from permutation group theory that are not defined in this paper, see \cite{DixonM1996}.  For a set $X$ we denote the symmetric and alternating groups on $X$ by $\Sym(X)$ and $\Alt(X)$.  If $\vert X\vert = n$ and the set is unimportant, we write $\Sym(n)$ and $\Alt(n)$.

\begin{defn}
Let $X$ and $Y$ be sets, $G\le \Sym(X)$, and $H\le Sym(Y)$.  Define the {\bf wreath product of $G$ and $H$}, denoted $G\wr H$, to be the set of all permutations of $X\times Y$ of the form $(x,y)\mapsto (g(x),h_x(y))$, where for each $x\in X$, $h_x$ is an element of $H$ that acts on $Y$, but for different $x \in X$, the choice of $h_x$ is independent.
\end{defn}

We caution the reader that these definitions of wreath products are not completely standard, in that some mathematicians use $H\wr G$ for what we have defined as $G\wr H$, and similarly use $\Gamma_2\wr \Gamma_1$ for our $\Gamma_1 \wr \Gamma_2$. Both orderings appear frequently in the literature.

We will often times consider wreath products of multiple digraphs or groups, and sometimes the specific digraphs or groups are unimportant.  In this circumstance, rather than write out and define the digraphs or groups, we will just say that a digraph $\Gamma$ or group $G$ is a {\bf multiwreath product}. Formally, a digraph $\Gamma$ is a multiwreath product if there exist nontrivial digraphs $\Gamma_1,\dotsc,\Gamma_r$ such that $\Gamma = \Gamma_1\wr\Gamma_2\wr\dotsb\wr\Gamma_r$, and a group $G$ is a multiwreath product if there exist nontrivial groups $G_1,\dotsc, G_r$ such that $G = G_1\wr G_2\wr\dotsb\wr G_r$.

Following \cite{Morris1999}, define a partial order on the set of abelian groups of order $p^n$ as follows:  We say $G\preceq _{p} H$  if there is a chain $H_1 < H_2 <\dotsb < H_m = H$ of subgroups of $H$ such that $H_1,H_2/H_1,\dotsc, H_m/H_{m-1}$
are all cyclic, and
$$G\cong H_1 \times \frac{H_2}{H_1}\times \dotsm \times \frac{H_m}{H_{m-1}}.$$
There is an equivalent definition for this partial order. Take the natural partial order on partitions of a fixed integer $n$, so for partitions $\mu: n = i_1+ \dotsb + i_m$ and $\nu: n =j_1+\dotsb+j_{m_0}$ of $n$, we say $\mu \le \nu$ if $\nu$ can be obtained from $\mu$ after possible rearrangement, by grouping some summands.
Now, $G \preceq_p H$ precisely if
$G\cong\Z_{p^{i_1}}\times\Z_{p^{i_2}}\times\dotsm\times\Z_{p^{i_m}}$ and $H\cong \Z_{p^{j_1}}\times\Z_{p^{j_2}}\times\dotsm\times\Z_{p^{j_{m_0}}}$ where $\mu \le \nu$.  In Figure \ref{poset} this partial order is depicted for abelian groups of order $p^5$.

\begin{figure}

\[
\xymatrix{
                              &    \Z_{p^5}                &                                 \\
\one{\Z_{p^4}\times\Z_p}{ur}  &                            & \one{\Z_{p^2}\times\Z_{p^3}}{ul}    \\
\three{\Z_{p^2}^2\times\Z_p}  &                            &  \four{\Z_{p^3}\times\Z_p^2}      \\
                              & \two{\Z_{p^2}\times\Z_p^3} &                                 \\
                              & \five{\Z_p^5}              &                                 \\
}
\]

\caption{The partial order for groups of order $p^5$}

\label{poset}

\end{figure}
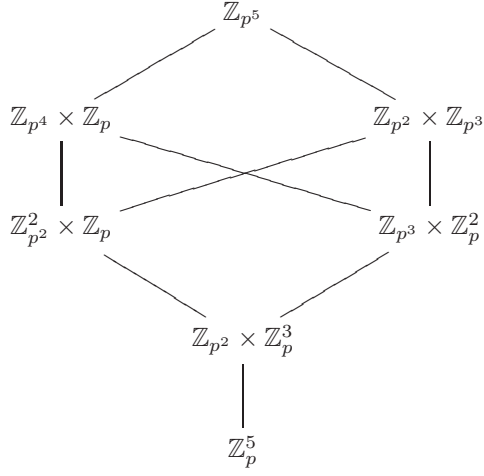

The following result was proven in \cite{Morris1999} (see also \cite{Morris1996}) in the case where $p$ is an odd prime and in \cite{KovacsS2012} when $p = 2$.

\begin{thm}\label{pktwogroups}
Let $\Gamma = \Cay(G,S)$ be a Cayley digraph on an abelian group $G$ of order $p^k$, where $p$ is prime. Then the following are equivalent:
\begin{enumerate}
\item The digraph $\Gamma$ is isomorphic to a Cayley digraph on both $\Z_{p^k}$ and $H$, where
$H$ is an abelian group with $\vert H\vert = p^k$, say $H = \Z_{p^{k_1}}\times\Z_{p^{k_2}}\times\dotsm\times\Z_{p^{k_{m_0}}}$, where $k_1 + \dotsb + k_{m_0} = k$.
\item There exist a chain of subgroups $G_1\le\dotsb\le G_{m-1}$ in $G$ such that
\begin{enumerate}
\item $G_1,G_2/G_1,\dotsc,G/G_{m-1}$ are cyclic groups;
\item $G_1\times G_2/G_1\times\dotsm\times G/G_{m-1} \preceq_p H$;
\item For all $s\in S\backslash G_i$, we have $sG_i\subseteq S$, for $i = 1,\dotsc,m - 1$. (That is, $S\backslash G_i$ is a union of cosets of $G_i$.)
\end{enumerate}
\item There exist Cayley digraphs $U_1,\dotsc,U_m$ on cyclic $p$-groups $H_1,\dotsc,H_m$ such
that $H_1\times\dotsm\times H_m \preceq_p H$ and $\Gamma\cong U_m\wr\dotsb\wr U_1$.
\end{enumerate}
Furthermore, any of these implies:
\begin{enumerate}
\item[4.]$\Gamma$ is isomorphic to Cayley digraphs on every abelian group of order $p^k$ that
is greater than $H$ in the partial order.
\end{enumerate}
\end{thm}

\section{Group Theoretic Results}

In the section we collect all permutation group theoretic results that we will need for our main result.

\begin{thm}[Theorem 3 of \cite{Jones2002}, or Corollary 1.2 of \cite{Li2003}]\label{dtcyclic}
A primitive permutation group $K$ acting on $\Omega$ of finite degree n has a cyclic regular
subgroup if and only if one of the following holds:
\begin{enumerate}
\item $\Z_p\le K\le \AGL(1,p)$, where $n=p$ is prime;
\item $K = \Sym(n)$ for some $n$, or $K = \Alt(n)$ for some odd $n$;
\item $\PGL(d,q)\le K\le \PGammaL(d,q)$ where $n = (q^d-1)/(q-1)$ for some $d\ge 2$;
\item $K = \PSL(2,11)$, $M_{11}$ or $M_{23}$ where $n = 11$, $11$ or $23$ respectively.
\end{enumerate}
\end{thm}

When we refer to any of the groups listed in the above theorem, we will be considering it not as an abstract group, but as a permutation group endowed with its natural action.

If $G$ has a block system ${\mathcal B}$ with blocks of minimal size, then the action of the set-wise stabilizer $H$ in $G$ of the block $B\in{\mathcal B}$ is primitive.  We will only be concerned with the case when $G$ contains a regular cyclic subgroup, and it is not hard to show that the action of $H$ on $B$ is one of the groups in Theorem \ref{dtcyclic}.  We now consider conjugation results concerning the groups in Theorem \ref{dtcyclic} that we will require later.

`The following result can be deduced from Theorem 1.1 and Corollary 1.2 of \cite{Li2003}.

\begin{lem}\label{Singer}
If $K$ satisfies $\PGL(d,q)\le K\le\PGammaL(d,q)$, then every regular abelian subgroup of $K$ is cyclic. Furthermore, any such subgroup is a Singer subgroup (and so any two are conjugate in $\PGL(d,q)$) unless $d = 2$ and $q = 8$, in which case $n = 9$.
\end{lem}

The conjugacy of regular cyclic subgroups is also noted in Corollary 2 of \cite{Jones2002}, but the fact that all regular abelian subgroups are cyclic is proved in the Li paper.

The following result is well-known.

\begin{lem}\label{altcon}
If two elements of $\Alt(n)$ are conjugate in $\Sym(n)$ but not in $\Alt(n)$, then their cycle structures are the same, they have no cycle of even length, and the lengths of all of their odd cycles are distinct.
\end{lem}

\begin{lem}\label{easyfact}
Let $K,K' \le G$ be conjugate in $G$.  If $N \tl G$ and $KN= G$, then there exists $n\in N$ with $n^{-1}Kn = K'$.

In particular, any two regular cyclic subgroups of $\PGL(d,q)$ are conjugate by an element of $\PSL(d,q)$.  Also, if $n$ is even, then any two regular cyclic subgroups of $\Sym(n)$ are conjugate by an element of $\Alt(n)$.
\end{lem}

\begin{proof}
Let $g\in G$ such that $g^{-1}Kg = K'$.  As $KN = G$ there exists $k\in K$, $n\in N$ such that $kn = g$.  Now $n^{-1}Kn=g^{-1}kKk^{-1}g = g^{-1}Kg=K'.$

By \cite[Corollary 2]{Jones2002} or \cite[Corollary 1.2]{Li2003}, there is one conjugacy class of regular cyclic subgroups of $\PGL(d,q)$, and clearly there is always one conjugacy class of regular cyclic subgroups in $\Sym(n)$.  If $K$ is a regular cyclic subgroup of $\PGL(d,q)$, then $K\PSL(d,q) = \PGL(d,q)$  by \cite[Lemma 2.3]{Li2003}. Similarly, if $K$ is a regular cyclic subgroup of $\Sym(n)$ where $n$ is even, then $K\Alt(n) = \Sym(n)$.  The result follows by the first paragraph of this proof, with $\PSL(d,q)$ or $\Alt(n)$ taking the role of $N$.
\end{proof}

The concept of $\Omega$-step imprimitivity will be important in this paper. Intuitively, on a set of cardinality $n$, the action of a transitive group is $\Omega$-step imprimitive if there is a sequence of nested block systems that is as long as possible (given $n$). The terms ``nested" and ``as long as possible" may not be clear, so we provide formal definitions below, including an explicit formula for $\Omega=\Omega(n)$.

\begin{defn}
Let $G$ be a transitive permutation group.  Let $Y$ be the
set of all block systems of $G$.  Define a partial order
on $Y$ by $\mathcal{B}\le\mathcal{C}$ if and only if every block of
$\mathcal{C}$ is a union of blocks of $\mathcal{B}$. We say that a strictly increasing sequence of $m+1$ block systems under this partial order is an {\bf \mathversion{bold}$m$-step imprimitivity sequence} admitted by $G$.
\end{defn}

An $m$-step imprimitivity sequence is what we referred to in our intuitive description as a ``nested" sequence.

\begin{defn}
Let $G$ be a transitive group of degree $n$.  Let $n = \Pi_{i=1}^rp_i^{a_i}$ be the prime factorization of $n$ and let ${\mathversion{bold}\Omega = \Omega(n)=\sum_{i=1}^ra_i}$. Then $G$ is {\mathversion{bold}\bf $\Omega$-step imprimitive} if it admits an $\Omega$-step imprimitivity sequence.

A block system $\mathcal{B}$ will
be said to be {\bf normal} if the elements of $\mathcal{B}$ are the orbits
of a normal subgroup.  We will say that $G$ is {\bf\mathversion{bold}normally
$\Omega$-step imprimitive} if $G$ is $\Omega$-step imprimitive with a sequence in which each block system is
normal.
\end{defn}

Let $\mathcal{B}_0 < \dotsb < \mathcal{B}_\Omega$ be an $\Omega$-step imprimitivity sequence of $G$, where $G$ is acting on $X$.  Then $\mathcal{B}_0$ consists of singleton sets, $\mathcal{B}_\Omega=\{X\}$, and if $B_i\in\mathcal{B}_i$ and $B_{i+1}\in\mathcal{B}_{i+1}$, then $\vert B_{i+1}\vert/\vert B_i\vert$ is a prime. Thus it is not possible to have a $k$-step imprimitivity sequence for any $k>\Omega$, satisfying our intuitive description of the system as being ``as long as possible".

Recall that we would like to characterise digraphs $\Gamma$ with $G, G' \le \Aut(\Gamma)$, where $G$ and $G'$ are regular (and nonisomorphic), and $G$ is cyclic. Our method will be to find a subgroup $N$ of $\Aut(\Gamma)$ that is normally $\Omega(n)$-step imprimitive. If we can then find regular subgroups $H,H'\le N$ with $H \cong G$ and $H'\cong G'$, then we will be able to use the fact that $H$ and $H'$ also admit the many block systems of $N$ to determine a lot about the structure of $\Gamma$.

In fact, in Lemma~\ref{tool1}, we will show that there is some conjugate $\delta^{-1}G'\delta\le \Aut(\Gamma)$ of $G'$, such that $K=\langle G, \delta^{-1}G'\delta\rangle$ is normally $\Omega(n)$-step imprimitive. Then in Theorem~\ref{maingroup}, we show that if we assume a numerical condition on $n$, there is a nilpotent group $N \le K$ that contains subgroups isomorphic to both $G$ and $G'$. Clearly $N$ is still $\Omega(n)$-step imprimitive since it must admit all of the blocks that $K$ admits.

The next two lemmas, the intervening corollary, and the definitions surrounding them are required for the proof of Lemma~\ref{tool1}, which is in turn used in the proof of Theorem~\ref{maingroup}.

\begin{defn}
Let $X$ be a set and $G\le Sym(X)$ be transitive with and $\mathcal{B}$ a block system of $G$.  For $g\in G$ we denote by $g/\mathcal{B}$ the permutation of $\mathcal{B}$ induced by $g$, and $G/\mathcal{B} = \{g/\mathcal{B}:g\in G\}$.

By {\mathversion{bold}$\fix_G(\mathcal{B})$} we mean the subgroup of $G$ which fixes each block of $\mathcal{B}$ set-wise.  That is, $$\fix_G(\mathcal{B}) = \{g\in G:g(B) = B{\rm\ for\ all\ }B\in\mathcal{B}\}.$$ This can also be thought of as the kernel of the projection from $G$ to $G/\mathcal B$.

For $B\in\mathcal{B}$, we denote the set-wise stabilizer of the block $B$ by {\mathversion{bold}$\Stab_G(B)$}.  That is, $$\Stab_G(B) = \{g\in G:g(B) = B\}.$$

The {\bf\mathversion{bold}support of $G$}, denoted $\supp(G)$, is the set of all $x\in X$ that are acted on nontrivially by some $g\in G$. That is, $$\supp(G) = \{x\in X:\text{there exists }g\in G \text{ such that }g(x) \neq x\}.$$
\end{defn}

\begin{lem}[Lemma 2.2 of \cite{Dobson2009}]\label{simplegroupblock}
Let $H \le \Sym(n)$ be transitive such that $H$ admits a complete block system $\mathcal B$. If $T=(\soc(\fix_H(\mathcal B)))^B$ is a transitive nonabelian simple group where $B \in \mathcal B$, then $\mathcal C=\{\supp(L): L \text{ is a minimal normal subgroup of } \fix_H(\mathcal B)\}$ is a complete block system of $H$, $ \mathcal B \le \mathcal C$, and $\soc(\fix_H(\mathcal B))$ is a direct product of simple groups isomorphic to $T$.
\end{lem}

Some of the arguments in the following corollary are also used in the proof of the above lemma, but since they are not included in the final statement, we repeat them for clarity.

\begin{cor}\label{simplegroupblock-cor}
Let $H$, $\mathcal B$, $T$, and $\mathcal C$ be as in the statement of Lemma~\ref{simplegroupblock}. If the blocks of $\mathcal C$ are $C_1, \ldots, C_k$, then we can write $\soc(\fix_H(\mathcal B)) =T_1 \times \cdots \times T_k$ where the support of $T_i$ is $C_i$ for every $1 \le i \le k$.

Let $C \in \mathcal C$ and let $B, B'$ be two blocks of $\mathcal B$ that lie inside $C$. If $T$ also has trivial centralizer in $\fix_H(\mathcal B)^B$ then there is no element $\gamma\in \fix_H(\mathcal B)$ such that $\gamma^B=1$ but $\gamma^{B'}\neq 1$.

In particular, this applies if $T$ is primitive on $B$.
\end{cor}

\begin{proof}
First notice that as Lemma~\ref{simplegroupblock} states, $\soc(\fix_H(\mathcal B))$ is a direct product of isomorphic finite simple groups isomorphic to $T$, so the kernel of any homomorphism from $\soc(\fix_H(\mathcal B))$ onto $T$ is the same as the kernel of some projection onto a single factor.

Thus, each block of $\mathcal B$ is in the support of a unique direct factor of $\soc(\fix_H(\mathcal B))$, so there is a well-defined map from $\mathcal B$ to the direct factors of $\soc(\fix_H(\mathcal B))$ determined by the direct factor that includes $\mathcal B$ in its support.

Now, taking any $h \in \fix_H(\mathcal B)$, since $h$ fixes every block in the support of any direct factor $T_1$ of  $\soc(\fix_H(\mathcal B))$, the map we've just mentioned means that $T_1$ is normalized by $h$. Thus, the direct factors in $\soc(\fix_H(\mathcal B))$, which are the minimal normal subgroups of $\soc(\fix_H(\mathcal B))$, are also normal subgroups (and therefore minimal normal subgroups) in $\fix_H(\mathcal B)$. In other words, there is a one-to-one correspondence between the blocks of $\mathcal C$ and the direct factors of $\soc(\fix_H(\mathcal B))$, establishing the first claim of this corollary.

Towards a contradiction, suppose that there exists $\gamma \in \fix_H(\mathcal B)$ and $B,B'\in{\mathcal B}$, $B\cup B'\subseteq C\in{\mathcal C}$,  such that $\gamma^B=1$ but $\gamma^{B'}\neq 1$. Let $N=\langle \gamma \rangle^{\fix_H(\mathcal B)}$ be the normal closure of $\langle \gamma\rangle$ in $\fix_H(\mathcal B)$, and notice that since $\gamma^B=1$ we also have $N^B=1$.

Let $T_{B'}=(\soc(\fix_H(\mathcal B)))^{B'}$. By hypothesis, $T_{B'}$ has trivial centralizer in $\fix_H(\mathcal B)^{B'}$, so there is some element $t' \in T_{B'}$ such that $[t',\gamma]$ is nontrivial on $B'$. Let $M$ be the direct factor of $\soc(\fix_H(\mathcal B))$ with $\supp(M)=C$.  As $C = \supp(M)$ and $B\subseteq C$, $M^B\not = 1$.  Then there exists $t \in M$ such that $t^B=t'$. Now, $[t,\gamma] \in N$, since $\gamma \in N$, $N \trianglelefteq \fix_H(\mathcal B)$, and $t \in \fix_H(\mathcal B)$. Also, $[t,\gamma] \in M$, since $t\in M$, $M \trianglelefteq \fix_H(\mathcal B)$, and $\gamma \in \fix_H(\mathcal B)$. So $N \cap M$ is a nontrivial normal subgroup of $M$. Since $M$ is simple, $N\ge M$.  But then $N^B = 1$ and so $M^B = 1$, a contradiction.

We complete this proof by showing that if $T$ is primitive then its centralizer in $\fix_H(\mathcal B)^B$ is trivial, so this corollary applies.  Note that a point-stabilizer in $T$ cannot be normal in $T$ because a nontrivial normal subgroup of a primitive group must be transitive.  Also, a point-stabilizer in a primitive group is maximal, so it must be self-normalizing. The result now follows by \cite[Theorem 4.2A (iv)]{DixonM1996}.
\end{proof}

\begin{lem}\label{doubly-transitive} Let $G, G'$ be regular abelian subgroups of a primitive group $K$  of degree $n$,  with $G$ cyclic.  Let $p$ be any prime divisor of $n$, let $G_p$ be the unique subgroup of $G$ of order $p$, and let $G'_p$ be any subgroup of $G'$ of order $p$. Then there exists $\delta\in \soc(K)$ such that $ G_p=\delta^{-1}G'_p\delta$.
\end{lem}

\begin{proof}
By our hypotheses, $K$ must be given in Theorem~\ref{dtcyclic}.

Suppose first that $n=p$ is prime (this deals with parts (1) and (4) of Theorem~\ref{dtcyclic}, as well as with some cases of part (2)). In this case, $G$ and $G'$ are Sylow $p$-subgroups of $K$. Furthermore, either $K\le \AGL(1,p)$ has a unique regular subgroup (this covers $p=2$ in particular), $K$ is simple and $\soc(K)=K$, or $K=\Sym(p)$ and $\soc(K)=\Alt(p)$ with $p>2$. In the first case, $G=G'=\soc(K)$ and the result is immediate. In the third case $p$ is odd and $G, G' \le \Alt(p)=\soc(K)$, so in both the second and third cases $G,G' \le \soc(K)$ are Sylow $p$-subgroups of $\soc(K)$ and we can use a Sylow theorem.

Next we suppose that $\soc(K) = \Alt(n)$.  Let $G_p\le G$, $G'_p\le G'$ be subgroups of order $p$.  Unless $p=2$ and $n/p$ is odd, the semiregular cyclic groups $G_p$ and $G'_p$ are generated by elements $g$ and $g'$ respectively, where $g, g' \in \Alt(n)$ each has a cycle structure consisting of $n/p>1$ cycles of length $p$, since $n$ is composite. By Lemma~\ref{altcon}, there is some $\delta \in \Alt(n)$ such that $g=\delta^{-1}g'\delta$, and so $G_p=\delta^{-1}G'_p\delta$. If $p=2$ and $n/p$ is odd, then the elements $g,g'$ are not in $\Alt(n)$.   As $g,g'\in K$, $K = \Sym(n)$ and there is some element $\gamma$ of  $\Sym(n)$ such that $\gamma^{-1}g\gamma=g'$.  This deals with part (2) of Theorem~\ref{dtcyclic}.

We may now assume that part (3) of Theorem~\ref{dtcyclic} holds, so $\PGL(d,q)\le K\le \PGammaL(d,q)$. Unless $d=2$, $q=8$, and $n=8$, we may apply Lemma~\ref{Singer} to deduce that both $G$ and $G'$ are cyclic Singer subgroups that are conjugate in $\PGL(d,q)$.  Furthermore, by Lemma~\ref{easyfact} there exists $\delta\in \soc(K)=\PSL(d,q)$ such that $\delta^{-1}G'\delta = G$, and we are done. 

To complete our proof, we need to address the possibility that $\PGL(2,8)\le K\le \PGammaL(2,8)$ so $n = 9$ and both $G$ and $G'$ are cyclic groups of order $9$.  Since the index of $\PGL(2,8)$ in $\PGammaL(2,8)$ is $3$, we have only these two possibilities for $K$. If $K = \PGL(2,8)$ (so its Sylow $3$-subgroups have order $9$), then every regular subgroup of $K$ is a Sylow subgroup and the result follows by a Sylow theorem.

As might be presumed from the exception in Lemma~\ref{Singer}, $\PGammaL(2,8)$ does contain more than one conjugacy class of regular cyclic subgroups. However, we will show that the regular subgroups of order $3$ in these regular subgroups are conjugate, completing the proof of this result. In order to do this, we begin by  claiming that the number of Sylow $3$-subgroups of $\PGL(2,3)$ (that is, the Singer cycles) is the same as the number of Sylow $3$-subgroups of $\PGammaL(2,8)$, and that each Singer cycle lies in a unique Sylow $3$-subgroup of $\PGammaL(2,8)$. We will need this to ensure that there is a conjugating element $\delta$ lying in $\PSL(2,8)=\PGL(2,8)$ (not just in $\PGammaL(2,8)$) that takes any given Sylow $3$-subgroup of $\PGammaL(2,8)$ to any other.

If $K = \PGammaL(2,8)$ (so its Sylow $3$-subgroups have order $27$), then we will first show that $N_{\PGL(2,8)}(S)$ has order $18$, where $S$ is a regular cyclic subgroup of $\PGL(2,8)$ of order $9$ and so generated by a Singer cycle.  The normalizer must contain $S$ and be a subgroup of the normalizer in $\Sym(9)$ of $S$ which has order $54=6\cdot 9$.  It is also a subgroup of $\PGL(2,8)$ which has order $2^3\cdot 3^2\cdot7$, so the only possibilities are $9$ or $18$. If this normalizer has order $9$, then $S$ is contained in the center of its normalizer (the normalizer being $S$ is abelian), and so by Burnside's Transfer Theorem \cite[Theorem 7.4.3]{Gorenstein1968} $\PGL(2,8)$ has a normal $3$-complement, and has blocks whose order is relatively prime to $3$, a contradiction.  So $N_{\PGL(2,8)}(S)$ has order $18$.

Now, the number of Sylow $3$-subgroups of $\PGL(2,8)$ is the index of $N_{\PGL(2,8)}(S)$ in $\PGL(2,8)$, which we can now calculate to be $28$. Let $P$ be a Sylow $3$-subgroup of $K$ that contains $S$. Since $P$ is a Sylow $3$-subgroup of $K$ we have $\vert P\vert = 27$.  If $P$ contains two Singer subgroups of $\PGL(2,8)$, then $P$ is generated by these Singer subgroups as $\vert P\vert = 27$, and so $P\le \PGL(2,8)$, a contradiction. So $P$ contains only one Singer subgroup of $\PGL(2,8)$, namely $S$. Thus, there must be at least $28$ Sylow $3$-subgroups of $K$. Now, using Sylow's theorems, the number of Sylow $3$-subgroups of $K$ is $1\pmod{3}$ and a divisor of $|K|=2^3\cdot 3^3\cdot 7$, so $28$ is the only possibility given our lower bound of $28$.  So there are $28$ Sylow $3$-subgroups of $\PGammaL(2,8)$ and each one contains a unique Singer subgroup.  This completes the proof of our claim.

Now let $G$ and $G'$ be as in our hypothesis. Let $P$ be the Sylow $3$-subgroup of $K$ that contains $G$, and let $Q$ be the Sylow $3$-subgroup of $K$ that contains $G'$. Let $S \le P$ and $T \le Q$ be the unique Singer cycles in $P$ and $Q$. We claim that $G_3$  is the unique subgroup of $S$ of order $3$. Since nothing in our hypotheses distinguishes $S$ from $T$ or $G$ from $G'$, this will imply that $G'_3$ is the unique subgroup of $T$ of order $3$, and therefore since there exists $\delta \in \PGL(2,8)=\PSL(2,8)$ such that $\delta^{-1}T\delta=S$, we have $\delta^{-1}G'_3\delta=G_3$. To prove our claim, it certainly suffices to show that $P$ itself has a unique semiregular subgroup of order $3$ (which is $G_p$); this is what we will do.

We may assume that $P =  \{x\mapsto ax + b:a \in \{1,4,7\},b\in\Z_9\}$ as the multiplicative orders of $4$ and $7$ are $3$ in $\Z_9$, and the subgroup $\{x\mapsto ax:a = 1,4,7\}$ normalizes $x\mapsto x + 1$ and $P$ has order $27$.  Let $H\le P$ be semiregular of order $3$, and note that such an $H$ exists as $H' = \{x\mapsto x + b:b\in\la 3\ra\}$ is semiregular of order $3$.  Let $f\in H$ with $\langle f \rangle = H$, where $f(x) = ax + b$, $a \in\{ 1,4,7\}$ and $b\in\Z_9$.  Then $f^3(x) = x+(a^2+a+1)b = x + 3b = x$, from which we conclude that $b\equiv 0\ (\mod 3)$.  Thus $b =  3i$, where $i \in \{0,1,2\}$.  Set $a = 1 + 3j$, where $j \in \{0,1,2\}$.  If $j = 0$, then $H = H'$.  Otherwise, let $\ell\in\Z_9$ such that $j\ell\equiv -i\ (\mod 3)$.  Then

\begin{eqnarray*}
f(\ell) = a\ell + b = (1 + 3j)\ell + 3i = \ell + 3j\ell + 3i \equiv \ell - 3i + 3i\ (\mod 9) = \ell.
\end{eqnarray*}
We conclude that $f$ has a fixed point and $H$ is not semiregular.  Hence there is a unique semiregular subgroup of $P$ as required. This completes the proof.
\end{proof}





Let $G\le \Sym(n)$ admit a block system $\mathcal{B}$.  It is straightforward to observe that there is a block system $\mathcal{D}$ in $G/\mathcal{B}$ if and only if there is a block system $\mathcal{C}\ge \mathcal B$ of $G$ where a block of $\mathcal{C}$ consists of the union of all blocks of $\mathcal{B}$ contained within a block of $\mathcal{D}$.

\begin{defn}
If $G$ admits block systems $\mathcal B$ and $\mathcal C$ with $\mathcal B \le \mathcal C$, then
 we denote the corresponding block system $\mathcal{D}$ in $G/\mathcal B$ by $\mathversion{bold}\mathcal{C}/\mathcal{B}$.
\end{defn}

We are now ready to show that we can find a group that contains our regular cyclic subgroup and a conjugate of any other given regular abelian subgroup, and is $\Omega(n)$-step imprimitive.  Although it is not immediately clear from the statements of the results as written, \cite[Theorem 4.9 (i)]{Muzychuk1999} is a consequence of this lemma. 

\begin{lem}\label{tool1}
Let $G, G'$ be regular abelian subgroups of a permutation group of degree $n$, with $G$ cyclic.  Let $n = p_1^{a_1}\dotsm p_r^{a_r}$ be the prime-power decomposition of $n$, and $\Omega = \Omega(n)=\sum_{i=1}^r a_i$.  Then there exists $\delta\in\la G,G'\ra$ such that $\la G,\delta^{-1}G'\delta\ra$ is normally $\Omega$-step imprimitive.
\end{lem}

\begin{proof}
We proceed by induction on $\Omega=\Omega(n)$.  If $\Omega(n) = 1$, then $n$ is prime and the result follows from a Sylow theorem.  Let $G, G',$ and $n$ satisfy the hypotheses, with $\Omega(n)\ge 2$.  Assume the result holds for all permutation groups of degree $n'$ with $\Omega(n')$ at most $\Omega(n) -1$.  If $H=\langle G,G'\rangle$ is primitive, then by Lemma \ref{doubly-transitive}, there exists $\delta\in H$ such that $\la G,\delta^{-1}G'\delta\ra$ has a nontrivial center and hence is imprimitive, so we may assume without loss of generality that $H$ is imprimitive.

Suppose that $\mathcal{B}$ is a block system of $H$ with $m$ blocks of size $k$.  Observe that since $G, G'$ are regular and abelian, the orbits of $\fix_H(\mathcal B)\triangleleft H$ are the blocks of $\mathcal B$, so $\mathcal B$ is a normal block system. We first show that if $k=p$ is prime, we can complete the proof; then we will devote the remainder of the proof to demonstrating that if $k$ is composite, then for any prime $p\mid k$ there exists $\delta\in H$ such that $\la G,\delta^{-1}G'\delta\ra$ admits a normal block system with blocks of prime size $p$. Replacing $G'$ by this conjugate and $\mathcal B$ by this system then completes the proof.

Suppose $k=p$ is prime, and set $\mathcal{B}_1 = \mathcal{B}$.  By the induction hypothesis, since $\Omega(n/p) = \Omega(n)-1=\Omega-1$, there exists $\delta\in \langle G/\mathcal B_1, G'/\mathcal B_1\rangle$ such that $\la G/\mathcal B_1,\delta^{-1}(G'/\mathcal B_1)\delta\ra$ is normally $(\Omega-1)$-step imprimitive with $(\Omega-1)$-step imprimitivity sequence $\mathcal{B}_1/\mathcal{B}_1 < \mathcal{B}_2/\mathcal{B}_1 < \dotsb < \mathcal{B}_\Omega/\mathcal{B}_1$.  Then taking $\delta_1\in H$ such that $\delta_1/\mathcal B_1=\delta$, we see that $\la G,\delta_1^{-1}G'\delta_1\ra$ is normally $\Omega(n)$-step imprimitive with $\mathcal{B}_0 < \mathcal{B}_1 < \dotsb < \mathcal{B}_m$. This completes the proof when $k=p$.

Suppose now that $k$ is composite. We assume that $k$ is chosen to be minimal, and so by \cite[Exercise 1.5.10]{DixonM1996} $\Stab_H(B)^B$ is primitive. Since $k$ is composite, for any block $B \in \mathcal B$, $\Stab_H(B)^B$ is doubly-transitive as $\Z_k$ is a Burnside group \cite[Theorem 3.5A]{DixonM1996}. Since the groups in Theorem~\ref{dtcyclic}(1) are not of composite degree, $\Stab_H(B)^B$ has nonabelian simple socle, $T_B$. In fact, $T_B \cong\PSL(d,q)$ for some $d,q$, or $T_B \cong \Alt(k)$, so $T_B$ is doubly-transitive.  Since $\fix_H(\mathcal{B})^B \triangleleft \Stab_H(B)^B$, we see that $(\fix_H(\mathcal{B})^B)\cap T_B$ is nontrivial and normal in $T_B$; since $T_B$ is a simple group, we conclude that $\soc(\fix_H(\mathcal B)^B)=T_B$.  Thus $\fix_H(\mathcal B)^B$ has a doubly-transitive socle, so must itself be doubly-transitive. Let $K=\soc(\fix_H(\mathcal B))$.  

We now show that $T_B=\soc(\fix_H(\mathcal B)^B)=(\soc(\fix_H(\mathcal B)))^B=K^B$; that is, we get the same group whether we restrict to $B$ before or after taking the socle. Clearly, $K= \soc(\fix_H(\mathcal B)) \tl \fix_H(\mathcal B)$, so $K^B \tl (\fix_H(\mathcal B))^B$. Therefore, some minimal normal subgroup of $(\fix_H(\mathcal B))^B$ lies in $K^B$. We have just shown that $\soc(\fix_H(\mathcal B)^B)=T_B$ is a nonabelian simple group, so $(\fix_H(\mathcal B))^B$ has only one minimal normal subgroup, meaning $T_B= \soc(\fix_H(\mathcal B)^B)\le K^B$. Now $K^B$ (since it contains $T_B$) must also contain a composite regular cyclic subgroup, so just as above it too has a nonabelian simple socle. Since any socle is a direct product of simple groups and $T_B$ is already (doubly) transitive on $B$, it is not possible that $K^B>T_B$.
Thus $K^B=T_B$. 

Since $K^B=T_B$, it is a transitive nonabelian simple group, so the hypotheses of Lemma~\ref{simplegroupblock} are satisfied. Using Lemma~\ref{simplegroupblock} together with the first conclusion of Corollary~\ref{simplegroupblock-cor} (which requires no additional assumptions, we may write $K=T_1 \times \ldots \times T_r$, where for each $i \in \{1, \ldots, r\}$ we have $T_i \cong T_B$ and $\supp(T_i)$ is a block of $H$. Furthermore if we take $C_i=\supp(T_i)$ and $\mathcal C=\{C_1, \ldots, C_r\}$ then $\mathcal B \le \mathcal C$ (so $r \mid m$).

Let $\mathcal{B} = \{B_{i,j}:1\le i\le r{\rm\ and\ }1\le j\le m/r\}$ be labelled so that $\cup_{j=1}^{m/r}B_{i,j} = C_i$.  Let $p$ be a prime divisor of $k$.  Let $G_p\le\fix_G(\mathcal{B})$ be the unique subgroup of $G$ of order $p$, and $G'_p\le\fix_{G'}(\mathcal{B})$ a subgroup of order $p$.  Certainly, $G_p, G'_p \le \fix_H(\mathcal B)$, so $G_p^B, (G'_p)^B \le \fix_H(\mathcal B)^B$ for any $B \in \mathcal B$.

For any $1 \le i \le r$, let $T_i$ take the role of $K$ in Lemma~\ref{doubly-transitive}. We know that $K \le \fix_H(\mathcal B)$ and that $T_i$ is acting primitively on each block $B_{i,j}$ in $C_i$, so the hypotheses of the lemma are satisfied.  The lemma tells us that there exists $\delta_i\in T_i$ such that $(\delta_i^{-1}G'_p\delta_i)^{B_{i,1}} = G_p^{B_{i,1}}$.  Let $\delta_i^{-1}G'_p\delta_i = \la h\ra$, and $G_p = \la g\ra$.  Then there exists $b_i\in\Z$ such that $(gh^{b_i})^{B_{i,1}} = 1$.  Since $T_B=K^B$ is doubly transitive and therefore primitive, the hypotheses of Corollary~\ref{simplegroupblock-cor} are also satisfied, so since $gh^{b_i} \in \fix_H(\mathcal B)$, it must be the case that for every $1\le j \le m/r$, $(gh^{b_i})^{B_{i,j}} = 1$, and so $(\delta^{-1}_iG'_p\delta_i)^{C_i} = (G_p)^{C_i}$.
Also, $(\delta_i^{-1}G'_p\delta_i)^{C_j}=(G'_p)^{C_j}$ for any $j \neq i$, since $\delta_i^{C_j}=1$.   So if $\delta=\Pi_{i=1}^r \delta_i$ then $\delta \in K$ and $\delta^{-1}G'_p\delta=G_p$ is a central subgroup of $\langle G, \delta^{-1}G'\delta\rangle$, whose orbits are blocks of size $p$.
Thus, replacing $G'$ with $\delta^{-1}G'\delta$, we assume without loss of generality that $k = p$.
\end{proof}

\begin{defn}
Let $K\le \Sym(n)$ contain a regular abelian subgroup $G$. We say that a subgroup $N$ with $G \le N \le K$ \textbf{\mathversion{bold}mimics every regular abelian subgroup of $K$}, if the following two statements are equivalent:
\begin{itemize}
\item a regular abelian  group $M\le \Sym(n)$ is contained in $K$; and
\item $N$ contains a regular abelian  subgroup isomorphic to $M$.
\end{itemize}
If in addition, the subgroup of $N$ isomorphic to $M$ is conjugate in $K$ to $M$, we say that  $N$ \textbf{\mathversion{bold}mimics by conjugation every regular abelian subgroup of $K$}.
\end{defn}


We first give a sufficient condition for a group $K$ to contain a nilpotent subgroup that mimics by conjugation every regular abelian subgroup of $K$.

\begin{lem}\label{otherpaperlem}
Let $G$ be a regular cyclic subgroup of $K$. Suppose that whenever $R\le K$ is a regular abelian subgroup, then there exists $\delta\in K$ such that $\la G,\delta^{-1}R\delta\ra$ is nilpotent.  If $N$ is a maximal nilpotent subgroup of $K$ that contains $G$, then $N$ mimics by conjugation every regular abelian subgroup of $K$.
\end{lem}

\begin{proof}
Let $N$ be a maximal nilpotent subgroup of $K$ that contains $G$, and $R\le K$ a regular abelian subgroup of $K$.  We will show that there exists $\delta \in K$ with $\delta^{-1}R\delta \le N$, which will establish the result.

Let $n = p_1^{a_1}\dotsm p_r^{a_r}$ be the prime-power decomposition of $n$.  Note that since $G$ is a regular cyclic group, for any $1\le i \le r$ it admits block systems $\mathcal B_i$ with blocks of size $p_i^{a_i}$ and $\mathcal C_i$ with blocks of size $n/p_i^{a_i}$, and if a group containing $G$ as a subgroup admits a block system whose blocks have one of these sizes, it must be one of these block systems.  As $N$ is nilpotent and $G \le N$, by \cite[Lemma 10]{Dobson2003a} for every $1\le i \le r$, $N$ admits $\mathcal{B}_i$ and $\mathcal{C}_i$ as normal block systems. Also, by hypothesis, there is some $\delta_1 \in K$ such that $\la G,\delta_1^{-1}R\delta_1\ra$ is nilpotent and contains $G$, so admits each $\mathcal{B}_i$ and $\mathcal{C}_i$.  Since we are only aiming for a conjugate in $K$, we can replace $R$ by $\delta_1^{-1}R\delta_1$, if necessary, to assume that  $\langle G, R \rangle$ is nilpotent and admits each $\mathcal B_i$ and $\mathcal C_i$. Then $\la R,N\ra = \la G,R,N\ra$ admits $\mathcal{B}_i$ and $\mathcal{C}_i$ as block systems, and so $\la R,N\ra\le \Pi_{i=1}^r \Sym(p_i^{a_i})$ in its natural action on, say $B_1 \times \dotsm \times B_r$ where $B_i \in \mathcal B_i$, by \cite[Lemma 10]{Dobson2003}.

Now, $R/\mathcal{C}_1$ and $N/\mathcal{C}_1$ are $p_1$-subgroups of $\Sym(p_1^{a_1})$, and so there exists $\omega_1\in \la R,N\ra$ such that $\omega_1^{-1}R\omega/\mathcal{C}_1$ and $N/\mathcal{C}_1$ are contained in the same Sylow $p_1$-subgroup.  Then $\la \omega^{-1}R\omega,N\ra/\mathcal{C}_1$ is a $p_1$-group.  Inductively suppose that $\la \omega_j^{-1}\dotsm \omega_1^{-1}R\omega_1\dotsm\omega_j,N\ra/\mathcal C_i$ is a $p_i$-group for each $1 \le i \le j <r$. Then there is some $\omega_{j+1} \in \la  \omega_j^{-1}\dotsm \omega_1^{-1}R\omega_1\dotsm\omega_j,N\ra$ such that (by conjugation of Sylow $p_{j+1}$-subgroups) $\la  \omega_{j+1}^{-1}\dotsm \omega_1^{-1}R\omega_1\dotsm\omega_{j+1},N\ra/\mathcal C_{j+1}$ is a
$p_{j+1}$-group. As $\omega_{j+1} \in \la  \omega_j^{-1}\dotsm \omega_1^{-1}R\omega_1\dotsm\omega_j,N\ra$,  our inductive hypothesis implies that for each $1 \le i \le j <r$, $\la \omega_{j+1}^{-1}\dotsm \omega_1^{-1}R\omega_1\dotsm\omega_{j+1},N\ra/{\mathcal C}_i$ is still a $p_i$-group, completing the induction. Let $\delta=\omega_1 \dotsm \omega_r \in K$, so that $\la\delta^{-1}R\delta,N\ra /\mathcal C_i$ is a $p_i$-group for every $1 \le i \le r$.
This shows that $\la \delta^{-1}R\delta,N\ra$ is a direct product of its Sylow $p$-subgroups, so is nilpotent. By the maximality of $N$, we have $\la \delta^{-1}R\delta,N\ra = N,$ meaning $ \delta^{-1}R\delta \le N$, as required.
\end{proof}

We now characterize those values of $n$ for which transitive groups of degree $n$ that contain a regular cyclic subgroup always have this extremely useful property.)

\begin{thm}\label{maingroup}
Let $k = p_1\dotsm p_r$ where the $p_i$ are prime distinct primes, and $n = p_1^{a_1}\dotsm p_r^{a_r}$.  Let $G$ be a regular cyclic group of degree $n$.  The condition $\gcd(k,\varphi(k))=1$ is both necessary and sufficient to guarantee that whenever $G\le K\le S_n$ there exists a nilpotent subgroup $N$, with $G\le N\le K$, that mimics by conjugation every regular abelian subgroup of $K$.
\end{thm}

\begin{proof}
Suppose $\gcd(k,\varphi(k)) = 1$.  Let $R\le K$ be a regular abelian group.  By Lemma \ref{tool1} there exists $\delta_1\in K$ such that $\la G,\delta_1^{-1}R\delta_1\ra$ is normally $\Omega(n)$-step imprimitive.  By \cite[Theorem 12]{Dobson2003a}, (applying this requires our hypothesis that $\gcd(k,\varphi(k))=1$), there exists $\delta_2\in\la G,\delta_1^{-1}R\delta_1\ra\le K$ such that $\la G,\delta_2^{-1}\delta_1^{-1}R\delta_1\delta_2\ra$ is nilpotent.  By Lemma \ref{otherpaperlem}, any maximal nilpotent subgroup $N$ that contains $G$ mimics by conjugation every regular abelian subgroup of $K$.

Conversely, we will show that if $\gcd(k,\varphi(k)) > 1$, then for every $a_1, \ldots, a_r$ with $n= p_1^{a_1}\dotsm p_r^{a_r}$, there exists a regular cyclic subgroup $G$ of degree $n$, a group $K\ge G,$ and a regular abelian group $R\le K$ such that for every conjugate $T$ of $R$ in $K$, $\la G,T\ra$ is not nilpotent.  First, if $n = k$, then $n$ is square-free and the only abelian group of order $n$ is the cyclic group of order $n$.  So every regular abelian subgroup is cyclic.  By P\'alfy's Theorem \cite[Theorem A]{Palfy1987} there exists $G\le K\le S_n$ and a regular cyclic subgroup $T\le K$ such that $G$ and $T$ are not conjugate in $K$.  Also, the only transitive nilpotent subgroups of $S_n$ are the regular cyclic subgroups, so for any $\delta\in K$, $\la G,\delta^{-1}T\delta\ra$ is nilpotent if and only if $\delta^{-1}T\delta = G$.  So the result follows in this case.  

Suppose $n\not = k$.  We have just shown that if $G_k\le S_k$ is a regular cyclic subgroup of degree $k$, there exists $K_k\le S_k$ and $T_k\le K_k$ a regular cyclic subgroup such that no conjugate of $T_k$ is contained in a nilpotent subgroup of $K_k$ that contains $G_k$.  We now use this to build the groups of degree $n$ that we require.

Choose $K_k$ and $T_k$ as in the previous paragraph.  Let $K = K_k\wr \Sym(n/k)$ acting on the same set that $G$ is acting on. By the Embedding Theorem \cite[Theorem 1.2.6]{Meldrum1995} we may ensure that the orbits of the copies of $\Sym(n/k)$ are the orbits of the semiregular subgroup of $G$ of order $n/k$, so that $G$ is a subgroup of $K$ and therefore $G_k \le K_k$.  Clearly, $K$ is a subgroup of $S_n$.  Suppose that $T$ is a regular abelian subgroup of $K$.  Since $K$ admits a block system ${\mathcal B}$ consisting of blocks of size $n/k$ formed by the orbits of the normal subgroup $1 \wr \Sym(n/k)$, these must also be blocks of the subgroup $T$.  Then $T/{\mathcal B}$ is square-free and so cyclic, and since $T_k \le K_k$ again using the Embedding Theorem, we may choose such a $T$ so that $T/{\mathcal B} = T_k$.  

Suppose there exists $\delta\in K$ with $\delta^{-1}T\delta\le N$ where $G\le N$ and $N$ is nilpotent.

We claim that $N$ admits a normal block system with blocks of size $t$ for every $t \mid n$ (where $n$ is the degree of $N$). It is not difficult to see that, as a $p$-group (where $p$ is prime) has nontrivial center it has an element of order $p$ in its center.  Since a nilpotent group is the direct product of its Sylow subgroups, this implies that any nilpotent group contains a central element of order $p$. The orbits of this element form a normal block system with blocks of size $p$. We use this idea inductively to prove our claim. Let $p_1, \ldots, p_s$ be the (not necessarily distinct) prime factors of $n$. We will show that $N$ has a normal block system consisting of blocks of size $p_1\cdots p_i$ for every $1 \le i \le s$. Since we can choose any order for the prime factors of $n$, this will prove our claim.

As we have just shown, $N$ admits a normal block system with blocks of size $p_1$, and there is a corresponding normal subgroup, say $N_1$, whose orbits are these blocks; this establishes our base case. Suppose now that $N$ admits a normal block system with blocks of size $p_1\ldots p_{i-1}$ for some $2 \le i \le s-1$, and there is a corresponding normal subgroup $N_{i-1}$ whose orbits are these blocks. Then $N/N_{i-1}$ is a nilpotent subgroup and by the same observation admits a block system with blocks of size $p_i$, and there is a corresponding normal subgroup $N_i/N_{i-1}$ whose orbits are these blocks. Now $N_i$ is normal in $N$, and its orbits are blocks of size $p_1\cdots p_i$. 

In particular, $N$ admits a normal block system $\mathcal C$ with blocks of size $n/k$. As $G\le N$ and a regular cyclic group has exactly one block system with blocks of size $n/k$, we see ${\mathcal C} = {\mathcal B}$.  Since $\delta \in K=K_k \wr \Sym(n/k)$, we see that $\delta/\mathcal B \in K_k$. We must have $(\delta/{\mathcal B})^{-1}T_k(\delta/{\mathcal B})=(\delta^{-1}T\delta)/\mathcal B$ and $G_k=G/\mathcal B$, but then $$\la G_k,(\delta/{\mathcal B})^{-1}T_k(\delta/{\mathcal B})\ra= \la G,\delta^{-1}T\delta\ra/{\mathcal B}\le N/{\mathcal B}\le K/{\mathcal B} = K_k.$$ Since $N/\mathcal B$ is nilpotent, this contradicts our choice of $K_k$ and $T_k$ and establishes the result.
\end{proof}

\begin{rem}\label{Palfy gen}
If, in P\'alfy's Theorem, we restrict the hypothesis to only cyclic groups of square-free order, then Theorem \ref{maingroup} generalizes this restricted form.  This follows as the only nilpotent group of square-free degree is a cyclic group, and the only abelian group is also a cyclic group.  So a nilpotent group of square-free degree mimics by conjugation every abelian group of square-free degree if and only if every two regular cyclic subgroups of $G$ are conjugate in $G$.
\end{rem}

\section{Graph Theoretic Results}

\begin{lem}\label{wreathdigraphauto}
Let $p$ be prime and $k_1,\dotsc,k_{j}$ be positive integers.  There exist circulant digraphs $\Gamma_1,\ldots,\Gamma_j$ such that $\Aut(\Gamma_1 \wr \dotsb \wr \Gamma_j) = \Z_{p^{k_1}}\wr\dotsb\wr\Z_{p^{k_{j}}}$.
\end{lem}

\begin{proof}
Let $\vec{D_{k}}$ denote the directed cycle of length $k$, which is a circulant digraph. When $k>2$,
$\vec{D_{k}}$ has automorphism group $\Z_{k}\not\cong \Sym(k)$. When $k=2$, $\Aut(K_2)=\Aut(\bar{K_2})=\Z_2=\Sym(2)$ and $K_2$ and its complement are also circulant (di)graphs.

For each $1 \le i \le j$, let
\begin{equation*}\Gamma_i=
\begin{cases}
\vec{D_{p^{k_i}}}\text{ if $p^{k_i}>2$;}\\
K_2 \text{ if $p^{k_i}=2$, and either $i=1$ or $\Gamma_{i-1}\neq K_2$; and}\\
\bar{K_2}\text{ otherwise.}
\end{cases}
\end{equation*}
Let $\Gamma=\Gamma_1 \wr \dotsb \wr \Gamma_j$. By \cite[Theorem 5.7]{DobsonM2009}, $\Aut(\Gamma)=\Z_{p^{k_1}}\wr\dotsb\wr\Z_{p^{k_{j}}}$.
\end{proof}

\begin{defn}
Let $G\le \Sym(X)$ be transitive, and ${\mathcal O}_1,\ldots,{\mathcal O}_r$ the orbits of $G$ in its natural action on $X\times X$.  The \textbf{orbital digraphs} of $G$ are the digraphs $\Gamma_i$ whose vertices are the elements of $X$ and arcs are ${\mathcal O}_r$, $1\le i\le r$.
\end{defn}

We now require Wielandt's notion of the 2-closure of a group.

\begin{defn}
Let $G\le \Sym(X)$ be transitive, and ${\mathcal O}_1,\ldots,{\mathcal O}_r$ the orbits of $G$ in its natural action on $X\times X$. The \textbf{\mathversion{bold}$2$-closure of $G$}, denoted $G^{(2)}$, is the largest group whose orbits on $X\times X$ are $\mathcal O_1, \ldots, \mathcal O_r$. We say that $G$ is \textbf{\mathversion{bold}$2$-closed} if $G=G^{(2)}$.
\end{defn}

It is easy to verify that $G^{(2)}$ is a subgroup of $\Sym(X)$ containing $G$ and, in fact, $G^{(2)}$ is the largest (with respect to inclusion) subgroup of $\Sym(X)$ that preserves every orbital digraph of $G$. Equivalently, $G^{(2)}$ is the automorphism group of the Cayley colour digraph obtained by assigning a unique colour to the edges of each orbital digraph of $G$. It follows that the automorphism group of a graph is $2$-closed.

\begin{cor}\label{wreathcor}
Let $p$ be prime and $G, H$ regular abelian groups of
degree $p^k$ with $G$ cyclic.  Then there exists $H' \le \langle G, H\rangle$ such that $H' \cong H$ is regular and
$(\la G,H'\ra)^{(2)} \cong \Z_{p^{k_1}}\wr\dotsb\wr \Z_{p^{k_m}}$ for some $k_1, \dotsc, k_m$ with $k_1+\dotsb+k_m=k$.
\end{cor}

\begin{proof}
Let $\Gamma$ be the Cayley colour digraph formed by assigning a unique colour to the edges of each orbital digraph of $\la G,H\ra$. Since $\Aut(\Gamma)$ has regular subgroups isomorphic to $G$ and to $H$, we see that $\Gamma$ is a Cayley colour digraph on both of these groups, similar to condition (1) of Theorem~\ref{pktwogroups}.

Although the theorem is not stated for colour digraphs, most of the proof involves only permutation groups, so it is not hard to see that the same result is true for colour digraphs.  By Theorem \ref{pktwogroups} (3), we can find $k_1',\dotsc, k_{m'}'$ such that $\Gamma\cong U_1\wr\dotsb \wr U_{m'}$ where each $U_i$ is a circulant colour digraph on $\Z_{p^{k_i'}}$, and $\Z_{p^{k_1'}}\times\dotsm\times\Z_{p^{k_{m'}'}}\preceq_p H$.
It is a standard and straightforward observation that $(\la G, H\ra)^{(2)}= \Aut(\Gamma) \ge \Aut(U_1)\wr \dotsb \wr \Aut(U_{m'}) \ge \Z_{p^{k_1'}}\wr \dotsb \wr\Z_{p^{k_{m'}'}}$. Call this last group $K$. By Lemma~\ref{wreathdigraphauto}, $K$ is $2$-closed.

Since $\Z_{p^{k_1'}}\times\dotsm\times\Z_{p^{k_{m'}'}}\preceq_p H$, the group $K$ contains a regular subgroup $H' \cong H$. Repeating the argument to this point, with $H'$ taking the role of $H$, we obtain $k_1, \dotsc, k_{m}$ with $K'=\Z_{p^{k_1}}\wr \dotsb \wr\Z_{p^{k_{m}}} \le (\la G, H'\ra)^{(2)}$ and $\Z_{p^{k_1}}\times\dotsm\times\Z_{p^{k_{m}}}\preceq_p H'$. Furthermore, since $G, H' \le K$ and $K$ is $2$-closed, we have $(\la G,H' \ra)^{(2)} \le K$. Thus $K' \le K$, which implies $\Z_{p^{k_1}}\times\dotsm\times\Z_{p^{k_{m}}}\preceq_p \Z_{p^{k_1'}}\times\dotsm\times\Z_{p^{k_{m'}'}}$. This in turn means $H' \le K'$. Now we have $G, H' \le K'$, and $K'$ is also $2$-closed by Lemma~\ref{wreathdigraphauto}, so $(\la G,H' \ra)^{(2)} \le K'$.
Hence $K'=(\la G, H'\ra)^{(2)}$.
\end{proof}

Given regular groups $G_1$ and $H_1$ of degree $a$ and $G_2$ and $H_2$ of degree $b$, there is an obvious method for constructing digraphs that are simultaneously Cayley digraphs of $G_1\times G_2$ and of $H_1\times H_2$.  Namely, construct a Cayley digraph $\Gamma_1$ of order $a$ that is a Cayley digraph of $G_1$ and $H_1$ and a digraph $\Gamma_2$ of order $b$ that is a Cayley digraph of $G_2$ and $H_2$, and then consider some sort of ``product construction" of $\Gamma_1$ and $\Gamma_2$ to produce a digraph $\Gamma$ of order $ab$ with $\Aut(\Gamma_1)\times\Aut(\Gamma_2)\le\Aut(\Gamma)$.  We write ``product construction" as there are two obvious products of $\Gamma_1$ and $\Gamma_2$ that ensure that $\Aut(\Gamma_1)\times\Aut(\Gamma_2)\le\Aut(\Gamma)$: the wreath product, and the Cartesian product.

\begin{defn}
Let $\Gamma_1,\dotsc,\Gamma_r$ be digraphs.  We say that $\Gamma$ is of {\bf product type $\Gamma_1,\dotsc,\Gamma_r$} if $\Aut(\Gamma_1)\times\dotsb\times\Aut(\Gamma_r)\le\Aut(\Gamma)$.
\end{defn}

\begin{thm}\label{main1}
Let $n = p_1^{a_1}\dotsm p_r^{a_r}$.  Let $G,H$ be regular abelian groups of degree $n$ with $G$ cyclic and let $G_{p_i}, H_{p_i}$ be Sylow $p_i$-subgroups of $G$ and $H$ respectively, and $\Gamma$ a Cayley digraph on $G$.
If
\begin{itemize}
\item[] $\Gamma$ is of product type $\Gamma_1,\dotsc,\Gamma_r$, where each $\Gamma_i$ is a Cayley digraph on both $G_{p_i}$ and a group isomorphic to $H_{p_i}$, $1\le i\le r$,
\end{itemize}
then
\begin{itemize}
\item[]
$\Gamma$ is isomorphic to a Cayley digraph of $H$.
\end{itemize}
Furthermore, the converse holds whenever
 $G,H \le N \le \Aut(\Gamma)$ for some nilpotent group $N$, and we can also conclude that each $\Aut(\Gamma_i)$ is a (possibly trivial) multiwreath product of cyclic groups.
\end{thm}

\begin{proof}
First suppose that $\Gamma$ is of product type $\Gamma_1,\dotsc,\Gamma_r$. Then $\Aut(\Gamma_1)\times\dotsm\times\Aut(\Gamma_r)\le\Aut(\Gamma)$. For each $i$, there is some $H'_{p_i} \cong H_{p_i}$ such that $H'_{p_i} \le \Aut(\Gamma_i)$. Thus $H\cong H'_{p_1}\times\dotsm\times H'_{p_r} \le \Aut(\Gamma)$, so $\Gamma$ is isomorphic to a Cayley digraph of $H$.

Conversely, suppose that $\Gamma$ is a Cayley digraph on $G$ that is also isomorphic to a Cayley digraph of the abelian group $H$.  Then $G$ and some $H'\cong H$ are regular subgroups of $\Aut(\Gamma)$.  By assumption, $G,H' \le N \le \Aut(\Gamma)$ for some nilpotent group $N$. Let $N = P_1 \times \dotsm\times P_r$ where $P_i$ is a Sylow $p_i$-subgroup of $N$.
Notice that for each $i$, $\la G_{p_i},H_{p_i}'\ra\le P_i,$ so we may choose $N$ so that $P_i=\langle G_{p_i},H_{p_i}\rangle$ for each $i$. Furthermore, since the $2$-closure of a $p$-group is a $p$-group \cite[Exercise 5.28]{Wielandt1969}, and $N^{(2)} = (P_1)^{(2)}\times\dotsm\times(P_r)^{(2)}$ \cite[Theorem 5.1]{Cameronetal2002} we see that $N^{(2)}$ is a nilpotent group that contains $G$ and $H'$, and since $\Aut(\Gamma)$ is $2$-closed we have $N^{(2)} \le \Aut(\Gamma)$.  We can therefore choose $N=\la G,H'\ra^{(2)} = (P_1)^{(2)}\times\dotsm\times(P_r)^{(2)}$.

Let $B_i$ be one of the orbits of $P_i^{(2)}$. As an orbit of a normal subgroup, $B_i$ is a block of $N$.
By Corollary \ref{wreathcor}, there exists $(H''_{p_i})^{B_i}\le \la G_{p_i}^{B_i},(H'_{p_i})^{B_i}\ra$ such that $(H''_{p_i})^{B_i}\cong (H'_{p_i})^{B_i}$ acts regularly, and the group $(\la (G_{p_i})^{B_i}, (H''_{p_i})^{B_i} \ra)^{(2)}$ is a multiwreath product of cyclic $p_i$-groups. 
By Lemma \ref{wreathdigraphauto} there exists a vertex-transitive digraph $\Gamma_i$ with $\Aut(\Gamma_i) = (\la G_{p_i}^{B_i}, (H''_{p_i})^{B_i} \ra)^{(2)}$. Now, $(\la G_{p_i}^{B_i}, (H''_{p_i})^{B_i} \ra)^{(2)} \le (P_i^{B_i})^{(2)}=P_i^{B_i}$ by Wielandt's Dissection Theorem  \cite[Theorem 6.5]{Wielandt1969}, so $\Aut(\Gamma_i) \le P_i^{B_i}$.  We claim that $\Aut(\Gamma_1)\times\dotsm\times\Aut(\Gamma_r)\le N\le \Aut(\Gamma)$, so that  $\Gamma$ is of product type $\Gamma_1,\dotsc,\Gamma_r$. This is because for any $i$, $N=P_i\times N_i'$, where $p_i \nmid |N_i'$, so we can define $\Aut(\Gamma_i)$ to act semiregularly by commuting with every element of $N_i'$, and in this form, $\Aut(\Gamma_i) \le P_i$ for each $i$.

To complete the proof, notice for each $i$ that since $\Aut(\Gamma_i)$ contains regular subgroups isomorphic to $G_{p_i}^{B_i}$ (which is cyclic) and to $(H''_{p_i})^{B_i} \cong (H'_{p_i})^{B_i}$, the digraph $\Gamma_i$ is a circulant digraph of order $p_i^{a_i}$ that is also a isomorphic to a Cayley digraph of $H_{p_i}$.
\end{proof}

Although the condition requiring a nilpotent group to achieve the converse in the above theorem may seem quite limiting, we observe that in Muzychuk's solution to the isomorphism problem for circulant digraphs \cite{Muzychuk2004}, he shows that when $G\cong H$ are cyclic groups, the isomorphism problem for general $n$ can be reduced to the prime power cases; this implies that $G$ and a conjugate of $H$ always lie in some nilpotent group together.

We point out in the coming corollary that for some values of $n$, the nilpotent group required to achieve the converse of the above theorem will always exist, even if $G$ and $H$ are not both cyclic.

\begin{cor}\label{main-cor}
Let $k = p_1\dotsm p_r$ be such that $\gcd(k,\varphi(k)) = 1$ where each $p_i$ is prime, and $n = p_1^{a_1}\dotsm p_r^{a_r}$.  Let $G,H$ be regular abelian groups of degree $n$ with $G$ cyclic and let $G_{p_i}, H_{p_i}$ be Sylow $p_i$-subgroups of $G$ and $H$ respectively, and $\Gamma$ a Cayley digraph on $G$.  Then $\Gamma$ is isomorphic to a Cayley digraph of $H$ if and only if $\Gamma$ is of product type $\Gamma_1,\dotsc,\Gamma_r$, where each $\Gamma_i$ is a Cayley digraph on both $G_{p_i}$ and a group isomorphic to $H_{p_i}$, $1\le i\le r$.
\end{cor}

\begin{proof}
The first implication is already proven in Theorem~\ref{main1}.

Conversely, suppose that $\Gamma$ is a Cayley digraph on $G$ that is also a Cayley digraph of the abelian group $H$.  Then $G$ and some $H'\cong H$ are regular subgroups of $\Aut(\Gamma)$.  By Theorem \ref{maingroup}, there exists a nilpotent subgroup $N\le\Aut(\Gamma)$ that contains $G$ and a regular subgroup $H'$ isomorphic to $H$. Replacing $H$ by $H'$ and applying Theorem~\ref{main1} yields the result. \end{proof}

\begin{defn}
Let $G$ and $H$ be abelian groups of order $n$, and for each prime $p_i\vert n$, denote a Sylow $p_i$-subgroup of $G$ or $H$ by $G_{p_i}$ or $H_{p_i}$, respectively.  Define a partial order $\preceq$ on the set of all abelian groups of order $n$ by $G\preceq H$ if and only if $G_{p_i}\preceq_{p_i} H_{p_i}$ for every prime divisor $p_i\vert n$.
\end{defn}

\begin{rem}\label{crucialremark}
If, in Theorem~\ref{main1}, $H$ is chosen to be minimal with respect to $\preceq$, then the rank of each Sylow $p_i$-subgroup of $H$ (i.e. the number of elements in any irredundant generating set) will be equal to the number of factors in the multiwreath product $\Aut(\Gamma_i)$. Moreover, $\Aut(\Gamma_1)\times\dotsm\times\Aut(\Gamma_r)$ will contain a regular subgroup isomorphic to the abelian group $R$ if and only if $\Aut(\Gamma)$ contains a regular subgroup isomorphic to the abelian group $R$.
\end{rem}

\begin{thm}\label{main}
Let $\Gamma = \Cay(G,S)$ for some abelian group $G$ of order $n$. Let $R,R'$ be regular abelian subgroups of $\Aut(\Gamma)$, with $R$ cyclic. Suppose that there is a nilpotent group $N$ with $R' \le N \le \Aut(\Gamma)$, such that $N$ mimics every regular abelian subgroup of $\Aut(\Gamma)$.  Then the following are equivalent:
\begin{enumerate}
\item The digraph $\Gamma$ is isomorphic to a Cayley digraph on both $R$ and $H$, where
$H$ is a regular abelian group; furthermore, if $H$ is chosen to be minimal with respect to the partial order $\preceq$ from amongst all regular abelian subgroups of $\Aut(\Gamma)$, then for each $i$, the Sylow $p_i$-subgroup of $H$ has rank $m_i$.
\item Let $P_i$ be a Sylow $p_i$-subgroup of $G$.  There exist a chain of subgroups $P_{i,1}\le\dotsb\le P_{i,m_i-1}$ in $P_i$ such that
\begin{enumerate}
\item $P_{i,1},P_{i,2}/P_{i,1},\dotsc,P_{i,m-1}/P_{i,m-2},P_i/P_{i,m_i-1}$ are cyclic $p_i$-groups;
\item $P_{i,1}\times P_{i,2}/P_{i,1}\times\dotsm\times, P_{i,m-1}/P_{i,m-2},P_i/P_{i,m_i-1}\preceq_{p} H_i$, where $H_i$ is a Sylow $p_i$-subgroup of $H$;
\item For all $s\in S\backslash (P_{i,j}\times G_i')$, we have $sP_{i,j}\subseteq S$, for $j = 1,\dotsc,m_i - 1$, where $G_i'$ is a Hall $p_i'$-subgroup of $G$ (of order $n/p_i^{a_i}$). That is, $S\backslash (P_{i,j}\times G_i')$ is a union of cosets of $P_{i,j}$.
\end{enumerate}
\item The digraph $\Gamma$ is of product type $\Gamma_1,\dotsc,\Gamma_r$, where each $\Gamma_i\cong U_{i,m_i}\wr\dotsb\wr U_{i,1}$ for some Cayley digraphs $U_{i,1},\dotsc,U_{i,m_i}$ on cyclic $p_i$-groups $K_{i,1},\dotsc,K_{i,m_i}$ such
that $K_{i,1}\times\dotsm\times K_{i,m_i} \preceq_{p} H_i$, where $H_i$ is a Sylow $p_i$-subgroup of $H$.
\end{enumerate}
Furthermore, any of these implies:
\begin{enumerate}
\item[4.]$\Gamma$ is isomorphic to Cayley digraphs on every abelian group of order $n$ that
is greater than $H$ in the partial order $\preceq$.
\end{enumerate}
\end{thm}

\begin{proof}
Throughout this proof, let $n=p_1^{a_1}\dotsm p_r^{a_r}$, where the $p_i$ are distinct primes.

(1)$\Rightarrow$(2):  By hypothesis there is a transitive nilpotent subgroup $N$ of $\Aut(\Gamma)$ that contains regular subgroups isomorphic to $G, R$, and $H$.  
By Theorem~\ref{main1}, there exist $\Gamma_1,\dotsc,\Gamma_r$ where each $\Gamma_i$ is a circulant digraph of order $p_i^{a_i}$, such that $\Gamma$ is of product type $\Gamma_1,\dotsc,\Gamma_r$, and each $\Aut(\Gamma_i)$ is a (possibly trivial) multiwreath product of cyclic groups.  Additionally, by Remark \ref{crucialremark}, we may assume that $R$ is a regular abelian subgroup of $\Aut(\Gamma)$ if and only if $\Aut(\Gamma_1)\times\dotsm\times\Aut(\Gamma_r)$ contains a regular abelian subgroup isomorphic to $R$.

Let $\Aut(\Gamma_i) = \Z_{p_i^{b_{m_i,i}}}\wr\dotsb\wr \Z_{p_i^{b_{1,i}}}$ (by Remark~\ref{crucialremark}, this multiwreath product does have $m_i$ factors). Then $\Aut(\Gamma_i)$ admits block systems $\mathcal{D}_{i,j}$, $1\le j\le m_i$ consisting of blocks of size $x_{i,j} = \Pi_{\ell = 1}^jp_i^{b_{\ell,i}}$.  Observe that $\Aut(\Gamma_i)/\mathcal{D}_{i,j} = \Z_{p_i^{b_{m_i,i}}}\wr\dotsb\wr \Z_{p_i^{b_{j+1,i}}}$ and $\fix_{\Aut(\Gamma_i)}(\mathcal{D}_{i,j})^D = \Z_{p_i^{b_{j,i}}}\wr\dotsb\wr \Z_{p_i^{b_{1,i}}}$, $D\in\mathcal{D}_{i,j}$.  Additionally, $\fix_{\Aut(\Gamma_i)}(\mathcal{D}_{i,j+1})/\mathcal{D}_{i,j}$ in its action on $D/\mathcal{D}_{i,j}\in\mathcal{D}_{i,j+1}/\mathcal{D}_{i,j}$ is cyclic of order $p_i^{b_{j+1,i}}$.  Let $P_i$ be a Sylow $p_i$-subgroup of $G$. Now, as $G$ is a transitive abelian group, $P_{i,j}=\fix_G(\mathcal{D}_{i,j})$ is semiregular and transitive on $D\in\mathcal{D}_{i,j}$.  As $\fix_{\Aut(\Gamma_i)}(\mathcal{D}_{i,j+1})/\mathcal{D}_{i,j}$ in its action on $D/\mathcal{D}_{i,j}\in\mathcal{D}_{i,j+1}/\mathcal{D}_{i,j}$ is cyclic of order $p_i^{b_{j+1,i}}$, we see that $P_{i,j+1}/P_{i,j}$ is cyclic of prime-power order, and for the same reason,
$$P_{i,1}\times P_{i,2}/P_{i,1}\times\dotsm\times P_{i,m-1}/P_{i,m-2}\times P_i/P_{i,m_i-1}\preceq_{p} H_i,$$ where $H_i$ is a Sylow $p_i$-subgroup of $H$.
For $1\le j\le m_i-1$, $N$ admits a block system $\mathcal{C}_{i,j}$ consisting of blocks of size $x_{i,j}\cdot n/p_i^{a_i}$, as well as a block system $\mathcal{B}_{i,j}$ consisting of blocks of size $x_{i,j}$, and of course $\mathcal{B}_{i,j}\le\mathcal{C}_{i,j}$.  
As $\Gamma$ is of product type $\Gamma_1,\dotsc,\Gamma_r$, and each $\Gamma_i$ is a circulant graph, we see that $\Aut(\Gamma_i) \times \Z_{n/p_i^{a_i}} \le \Aut(\Gamma)$; furthermore, $\Aut(\Gamma_i) \ge \Z_{p_i^{a_i}/x_{i,j}}\wr \Z_{x_{i,j}}$. Thus, $P_{i,j}\vert_C \le \Aut(\Gamma)$ for every $C \in \mathcal C_{i,j}$. Thus we see that between blocks $C,C'\in\mathcal{C}_i$, we have either every directed edge from a block of $\mathcal{B}_{i,j}$ contained in $C$ to a block of $\mathcal{B}_{i,j}$ contained in $C'$ or no directed edges. As $\mathcal{C}_i$ is formed by the orbits of $P_{i,j}\times G_i'$ and $\mathcal{B}_{i,j}$ is formed by the orbits of $P_{i,j}$, (2) follows.

$(2)\Rightarrow (3)$
Let $\mathcal{B}_{i,j}$ be the block system of $G$ formed by the orbits of $P_{i,j}$, $0\le j\le m_i$, and $\mathcal{C}_{i,j}$ the block system of $G$ formed by the orbits of $P_{i,j}\times G_i'$.  As for all $s\in S\backslash (P_{i,j}\times G_i')$, we have $sP_{i,j}\subseteq S$, for $j = 1,\dotsc,m_i - 1$, we have $\fix_G(\mathcal{B}_{i,j})\vert_{C_{i,j}}\le\Aut(\Gamma)$ for every $C_{i,j}\in\mathcal{C}_{i,j}$.  Note that $\fix_G(\mathcal{B}_{i,m_i}) = P_i$ and that $\la  \Z_{p_i^{a_i}},\fix_G(\mathcal{B}_{i,j})\vert_{C_{i,j}}:C_{i,j}\in\mathcal{C}_{i,j}\ra$ in its action on $B_{i,m_i}\in\mathcal{B}_{i,m_i}$ is $\Z_{p_i^{a_i - \ell_{i,j}}}\wr\Z_{p_i^{\ell_{i,j}}}$ where the orbits of $P_{i,j}$ have order $p_i^{\ell_{i,j}}$.  Now let $n_{i,j+1} = \ell_{i,j+1}-\ell_{i,j}$, $0\le j\le m_i - 1$.  Then $Q_i = \la\fix_G(\mathcal{B}_{i,j})\vert_{C_{i,j}}:C_{i,j}\in\mathcal{C}_{i,j}, 1\le j\le m_i \ra$ in its action on $B_{i,m_i}\in\mathcal{B}_{i,m_i}$ is $\Z_{p_i^{n_{i,m_i}}}\wr \Z_{p_i^{n_{i,m_i-1}}}\wr\dotsb\wr\Z_{p_i^{n_{i,1}}}$, and $Q_i = (\Z_{p_i^{n_{i,m_i}}}\wr \Z_{p_i^{n_{i,m_i-1}}}\wr\dotsb\wr\Z_{p_i^{n_{i,1}}})\times 1_{\Sym(n/p_i^{a_i})}.$
Let $K_{i,j}=\Z_{p_i^{n_{i,j}}}$ for each $1 \le j \le m_i$. Clearly, these are cyclic $p_i$-groups. Also, $P_{i,j+1}/P_{i,j} \cong K_{i,j+1}$ as abstract groups for $0 \le j \le m_i-1$, so assumption (b) tells us that $K_{i,1} \times \dotsm \times K_{i,m_i} \preceq_p H_i$.
Notice that $Q_1\times Q_2\times\dotsm \times Q_r\le\Aut(\Gamma)$, and by Lemma \ref{wreathdigraphauto} there exists a digraph $\Gamma_i$ with $\Aut(\Gamma_i) = Q_i$. The graphs $U_{i,j}$ are given in the proof of Lemma~\ref{wreathdigraphauto}. Thus $\Gamma$ is of product type $\Gamma_1, \dotsc, \Gamma_r$, completing this part of the proof.

$(3)\Rightarrow (1),(4)$ For each $i$, $\Aut(\Gamma_i)$ is a multiwreath product of $m_i$ Cayley digraphs on cyclic groups. It is straightforward to verify (or is an immediate consequence of the Universal Embedding Theorem) that $\Z_a\wr\Z_b$ contains regular subgroups isomorphic to both $\Z_{ab}$ and $\Z_a\times\Z_b$.

By assumption, $\Aut(\Gamma_i)\ge \Z_{p_i^{n_{i,m_i}}}\wr \dotsb \wr \Z_{p_i^{n_{i,1}}}$. Hence the above facts
tell us that $H_i=\Z_{p_i^{n_{i,1}}}\times \dotsm \times \Z_{p_i^{n_{m_i}}} \le \Aut(\Gamma_i)$, and also
that $H'_i \le \Aut(\Gamma_i)$ for every $H'_i \succ_p H_i$. Therefore $H=H_1 \times \dotsm \times H_r
\le \Aut(\Gamma)$, and also $H'\le \Aut(\Gamma)$ for every $H' \succ H$.
\end{proof}



The following result is obtained from the previous result by applying Theorem \ref{maingroup}.

\begin{cor}\label{mainresult}
Let $k = p_1\dotsm p_r$ be such that $\gcd(k,\varphi(k)) = 1$ where each $p_i$ is prime, and $n = p_1^{a_1}\dotsm p_r^{a_r}$. Let $\Gamma = \Cay(G,S)$ for some abelian group $G$ of order $n$. Then the following are equivalent:
\begin{enumerate}
\item The digraph $\Gamma$ is isomorphic to a Cayley digraph on both $\Z_n$ and $H$, where
$H$ is a regular abelian group; furthermore, if $H$ is chosen to be minimal with respect to the partial order $\preceq$ from amongst all regular abelian subgroups of $\Aut(\Gamma)$, then for each $i$, the Sylow $p_i$-subgroup of $H$ has rank $m_i$.
\item Let $P_i$ be a Sylow $p_i$-subgroup of $G$.  There exist a chain of subgroups $P_{i,1}\le\dotsb\le P_{i,m_i-1}$ in $P_i$ such that
\begin{enumerate}
\item $P_{i,1},P_{i,2}/P_{i,1},\dotsc,P_{i,m-1}/P_{i,m-2},P_i/P_{i,m_i-1}$ are cyclic $p_i$-groups;
\item $P_{i,1}\times P_{i,2}/P_{i,1}\times\dotsm\times P_{i,m-1}/P_{i,m-2} \times P_i/P_{i,m_i-1}\preceq_{p} H_i$, where $H_i$ is a Sylow $p_i$-subgroup of $H$;
\item For all $s\in S\backslash (P_{i,j}\times G_i')$, we have $sP_{i,j}\subseteq S$, for $j = 1,\dotsc,m_i - 1$, where $G_i'$ is a Hall $p_i'$-subgroup of $G$ (of order $n/p_i^{a_i}$). That is, $S\backslash (P_{i,j}\times G_i')$ is a union of cosets of $P_{i,j}$.
\end{enumerate}
\item The digraph $\Gamma$ is of product type $\Gamma_1,\dotsc,\Gamma_r$, where each $\Gamma_i\cong U_{i,m_i}\wr\dotsb\wr U_{i,1}$ for some Cayley digraphs $U_{i,1},\dotsc,U_{i,m_i}$ on cyclic $p_i$-groups $K_{i,1},\dotsc,K_{i,m_i}$ such
that $K_{i,1}\times\dotsm\times K_{i,m_i} \preceq_{p} H_i$, where $H_i$ is a Sylow $p_i$-subgroup of $H$.
\end{enumerate}
Furthermore, any of these implies:
\begin{enumerate}
\item[4.] $\Gamma$ is isomorphic to Cayley digraphs on every abelian group of order $n$ that
is greater than $H$ in the partial order $\preceq$.
\end{enumerate}
\end{cor}

A point about the previous results should be emphasized.  That is, it is necessary to introduce the digraphs $\Gamma_1,\dotsc,\Gamma_r$ from Lemma \ref{wreathdigraphauto} -  one cannot simply define $\Gamma_i = \Gamma[B_i]$ (the induced subgraph on the points of $B_i$), where $B_i\in\mathcal{B}_i$ and $\mathcal B_i$ is a block system whose blocks are the orbits of some Sylow $p$-subgroup of the regular cyclic subgroup.  Rephrased, it is possible for $\Gamma[B_i]$ to be a Cayley digraph of more than one group even when $\Gamma$ is only isomorphic to a Cayley digraph of a cyclic group, and even when the condition $\gcd(k,\varphi(k))=1$ is met.  We give an example of such a digraph in the following result.

\begin{eg}
Let $p$ and $q$ be distinct primes such that $\gcd(pq,\varphi(pq)) = 1$, and $\Gamma = \Cay(\Z_{p^2}\times\Z_{q},S)$, where $S = \{(kp,0),(1,1):k\in\Z_p\}$.  The only abelian group $H$ of order $p^2q$ for which $\Gamma$ is isomorphic to a Cayley digraph of $H$ is the cyclic group of order $p^2q$.  Nonetheless, let $\mathcal{B}$ be the block system of the left regular representation of $\Z_{p^2}\times\Z_{q}$, that has blocks of size $p^2$.  Then for every $B\in\mathcal{B}$, the induced subdigraph $\Gamma[B]$ is a Cayley digraph on $\Z_{p^2}$ and on $\Z_p^2$, and is isomorphic to the wreath product of two circulant digraphs of order $p$.
\end{eg}

\begin{proof}
Towards a contradiction, suppose that $\Gamma$ is isomorphic to a Cayley digraph of the abelian group $H'$, where $H'$ is not cyclic.  Let $G$ be the left regular representation of $\Z_{p^2}\times\Z_{q}$.  Then $H' = \Z_p^2\times\Z_q$, and by Theorem \ref{maingroup} there exists $H \cong H'$ such that $N = \la G,H\ra$ is nilpotent.  Then $N$ admits $\mathcal{B}$ as a block system as well as block systems $\mathcal{B}_p$ and $\mathcal{B}_q$ consisting of blocks of size $p$ and blocks of size $q$, respectively.  As $H\cong\Z_p^2\times\Z_q$, $N/\mathcal{B}_q$ is a $p$-group that contains regular subgroups isomorphic to $\Z_p^2$ and $\Z_{p^2}$.  By \cite[Lemma 4]{DobsonW2002} we have that $N/\mathcal{B}_q\cong\Z_p\wr\Z_p$.

For $0 \le i \le p^2-1$, let $B_{i,q} \in \mathcal B_q$ denote the block that consists of $\{(i,j):j \in \Z_q\}$. Each vertex of $B_{0,q}$ is at the start of a unique directed path in $\Gamma$ (that does not include digons) of length $q$ (travelling by arcs that come from $(1,1) \in S$), and each of these paths ends at a vertex of $B_{q,q}$. Thus any automorphism of $\Gamma$ that fixes $B_{0,q}$ must also fix $B_{q,q}$, contradicting $K/\mathcal{B}_q\cong\Z_p\wr\Z_p$.

Finally, $\Gamma[B]$ is isomorphic to $\Cay(\Z_{p^2},\{kp:k\in\Z_p^*\})\cong\bar{K}_p\wr K_p$, so by Theorem~\ref{pktwogroups}, $\Gamma[B]$ is also a Cayley graph on $\Z_p^2$.
\end{proof}

The converse though is true.  That is, if the condition $\gcd(k,\varphi(k))=1$ is met, and $\Gamma$ is a Cayley digraph of two abelian groups $G$ and $H$ with nonisomorphic Sylow $p$-subgroups $G_p$ and $H_p$, respectively, and $G$ is cyclic, then it must be the case that $\Gamma[B]$ is a wreath product, where $\mathcal{B}$ is formed by the orbits of $G_p$ and $B\in\mathcal{B}$.

\begin{cor}
Let $k = p_1\dotsm p_r$ be such that $\gcd(k,\varphi(k)) = 1$ where each $p_i$ is prime, and $n = p_1^{a_1}\dotsm p_r^{a_r}$.  Let $\Gamma$ be a circulant graph of order $n$, and let $\mathcal{B}_i$ be the block system of the left regular representation of $\Z_n$ consisting of blocks of size $p_i^{a_i}$.  If $\Gamma[B_i]$, $B_i\in\mathcal{B}_i$, is not a nontrivial wreath product and $H$ is an abelian group of order $n$ such that $\Gamma$ is isomorphic to a Cayley digraph of $H$, then a Sylow $p_i$-subgroup of $H$ is cyclic.  Consequently, if $\Gamma[B_i]$, $B_i\in\mathcal{B}_i$, is not isomorphic to a nontrivial wreath product for any $1\le i\le r$ then $\Gamma$ is not isomorphic to a Cayley digraph of any noncyclic abelian group.
\end{cor}

\begin{proof}
By Theorem~\ref{maingroup}, there is a transitive nilpotent subgroup $N$ of $\Aut(\Gamma)$ that contains the left regular representation of $\Z_n$ as well as a regular subgroup isomorphic to $H$. The system $\mathcal B_i$ is a block system of $N$ also, since $N$ is nilpotent. Thus $\mathcal B_i$ is a block system of $H$. Let $H_i$ denote a Sylow $p_i$-subgroup of $H$, and $G_i$ a Sylow $p_i$-subgroup of $\Z_n$. If  $H_i$ is not cyclic, then the restrictions of $H_i$ and $G_i$ to any $B_i \in \mathcal B_i$ are nonisomorphic regular $p$-groups, so by Theorem~\ref{pktwogroups}, $\Gamma[B_i]$ is a nontrivial wreath product.
\end{proof}

\section{Future Work}

The work in this paper provides a ``template" that one can use to approach the problem of when a digraph is a Cayley digraph of two nonisomorphic nilpotent groups, as follows.

Let $R,R'$ be two regular nilpotent groups of order $n$.  If there exists $\delta\in\la R,R'\ra$ such that $\la R,\delta^{-1}R'\delta\ra$ is nilpotent, then $\la R,\delta^{-1}R'\delta\ra^{(2)}$ is also nilpotent, and writing the group $\la R,\delta^{-1}R'\delta\ra^{(2)}$ as $\Pi_{i=1}^rP_i$, where $P_1,\dotsc,P_r$ are all Sylow subgroups of $\la R,\delta^{-1}R'\delta\ra^{(2)}$, then each $P_i$ is $2$-closed.  Furthermore, if $R_{p_i}$ is the Sylow $p_i$-subgroup of $R$, $R'_{p_i}$ is the Sylow $p_i$-subgroup of $R'$, and $P_i$ is the Sylow $p_i$-subgroup of $\la R, \delta^{-1}R'\delta\ra^{(2)}$, then $P_i=\la R_{p_i}, \delta^{-1}R'_{p_i}\delta\ra^{(2)}$. Thus, the subgraph induced on each orbit of $P_i$ is a Cayley graph on the Sylow $p_i$-subgroups of both $R$ and $R'$. So from a group theoretic point of view, this ``reduces" the group theoretic characterization to the corresponding prime-power cases.

Of course, we would ideally like conditions on the connection set of a Cayley digraph of one group to be a Cayley digraph of another group, but at this time such conditions are only known in the prime-power case when one of the groups is cyclic.  We also suspect that the conditions given in Theorem~\ref{main} and Corollary~\ref{mainresult} only hold when one of the groups is a cyclic group, and different conditions will be needed for different choices of nilpotent or even abelian groups.  So we have the following problem:

\begin{prob}
Given $p$-groups $P$ and $P'$, determine necessary and sufficient conditions on $S\subseteq P$ so that $\Cay(P,S)$ is also isomorphic to a Cayley digraph of $P'$.
\end{prob}

Some attempt to study this has been made in \cite{MMV}. The authors do not study the connection sets, but show that the situation is vastly different when neither regular subgroup in the automorphism group is cyclic, in the following sense. Theorem~\ref{pktwogroups} above shows that when one of the regular subgroups is cyclic (say of order $p^k$), the automorphism group of the graph is a multiwreath product, so the regular subgroup has index at least $p^{(k-1)(p-1)}$  in its automorphism group. However, Theorem 1.1 of \cite{MMV} shows that if $k \ge 3$ then given any non-cyclic abelian group of order $p^k$ where $p$ is odd, there is a Cayley digraph on that group whose automorphism group has order just $p^{k+1}$ (so the regular subgroup has index $p$ in this group) that contains a regular nonabelian subgroup also.

So how can one determine if a $\delta$ that conjugates $R'$ to lie in a nilpotent group with $R$ exists?  We have seen in this proof that sometimes in order to prove that they lie together in a nilpotent group, it is sufficient to show that they lie together in a group that is normally $\Omega(n)$-step imprimitive. Naively following the structure of the proof in this paper, we have the following problem:

\begin{prob}\label{prob2}
Determine for which regular nilpotent groups $N$ and $N'$ of order $n$ there exists $\delta\in \la N,N'\ra$ such that $\la N,\delta^{-1}N'\delta\ra$ is (normally) $\Omega(n)$-step imprimitive.
\end{prob}

This condition is certainly a necessary condition for a nilpotent subgroup of $\la N,N'\ra$ to contain $N$ and a conjugate of $N'$, but this condition is also sufficient under the arithmetic condition in Theorem \ref{mainresult} by \cite[Corollary 15]{Dobson2014}.  The solution of Problem \ref{prob2} will likely depend, like the proof of Theorem~\ref{dtcyclic} and consequently Theorem~\ref{maingroup}, on the Classification of the Finite Simple Groups.  In particular, it seems likely we will need at least a list of primitive groups which contain a regular nilpotent subgroup, as well as a perhaps a list of such nilpotent subgroups.  It is worthwhile to point out that Liebeck, Praeger, and Saxl have determined all primitive almost simple groups which contain a regular subgroup \cite[Theorem 1.1]{LiebeckPS2010}.

Finally, we conjecture that the condition $\gcd(k,\varphi(k)) = 1$ in Theorem \ref{mainresult} is unnecessary:

\begin{conj}
Theorem \ref{mainresult} holds for all $n\in {\mathbb N}$.
\end{conj}

\noindent Settling this conjecture may be quite difficult, as, at least with our approach, a positive solution would require generalizing Muzychuk's solution to the isomorphism problem for circulant color digraphs in \cite{Muzychuk2004}, which is a difficult and significant result.


\begin{thebibliography}{10}

\bibitem{Cameronetal2002}
Peter~J. Cameron, Michael Giudici, Gareth~A. Jones, William~M. Kantor,
  Mikhail~H. Klin, Dragan Maru{\v{s}}i{\v{c}}, and Lewis~A. Nowitz,
  \emph{Transitive permutation groups without semiregular subgroups}, J. London
  Math. Soc. (2) \textbf{66} (2002), no.~2, 325--333. \MR{MR1920405
  (2003f:20001)}

\bibitem{DixonM1996}
John~D. Dixon and Brian Mortimer, \emph{Permutation groups}, Graduate Texts in
  Mathematics, vol. 163, Springer-Verlag, New York, 1996. \MR{MR1409812
  (98m:20003)}

\bibitem{Dobson2000a}
Edward Dobson, \emph{On solvable groups and circulant graphs}, European J.
  Combin. \textbf{21} (2000), no.~7, 881--885. \MR{MR1787902 (2001j:05065)}

\bibitem{Dobson2003a}
\bysame, \emph{On isomorphisms of abelian {C}ayley objects of certain orders},
  Discrete Math. \textbf{266} (2003), no.~1-3, 203--215, The 18th British
  Combinatorial Conference (Brighton, 2001). \MR{MR1991717 (2004m:20006)}

\bibitem{Dobson2003}
\bysame, \emph{On the {C}ayley isomorphism problem for ternary relational
  structures}, J. Combin. Theory Ser. A \textbf{101} (2003), no.~2, 225--248.
  \MR{MR1961544 (2004c:05088)}

\bibitem{Dobson2006a}
\bysame, \emph{Automorphism groups of metacirculant graphs of order a product
  of two distinct primes}, Combin. Probab. Comput. \textbf{15} (2006), no.~1-2,
  105--130. \MR{MR2195578 (2006m:05108)}

\bibitem{Dobson2008}
\bysame, \emph{On solvable groups and {C}ayley graphs}, J. Combin. Theory Ser.
  B \textbf{98} (2008), no.~6, 1193--1214. \MR{MR2462314}

\bibitem{Dobson2009}
\bysame, \emph{On overgroups of regular abelian {$p$}-groups}, Ars Math.
  Contemp. \textbf{2} (2009), no.~1, 59--76. \MR{MR2491632}

\bibitem{Dobson2014}
\bysame, \emph{On the {C}ayley isomorphism problem for {C}ayley objects of
  nilpotent groups of some orders}, Electron. J. Combin. \textbf{21} (2014),
  no.~3, Paper 3.8, 15. \MR{3262245}

\bibitem{DobsonM2011}
Edward Dobson and Dragan Maru\v{s}i\v{c}, \emph{On semiregular elements of
  solvable groups}, Comm. Algebra \textbf{39} (2011), no.~4, 1413--1426.
  \MR{2804682}

\bibitem{DobsonM2009}
Edward Dobson and Joy Morris, \emph{Automorphism groups of wreath product
  digraphs}, Electron. J. Combin. \textbf{16} (2009), no.~1, Research Paper 17,
  30. \MR{MR2475540}

\bibitem{DobsonW2002}
Edward Dobson and Dave Witte, \emph{Transitive permutation groups of
  prime-squared degree}, J. Algebraic Combin. \textbf{16} (2002), no.~1,
  43--69. \MR{MR1941984 (2004c:20007)}

\bibitem{DobsonS2017}
Ted Dobson and Pablo Spiga, \emph{Cayley numbers with arbitrarily many distinct
  prime factors}, J. Combin. Theory Ser. B \textbf{122} (2017), 301--310.
  \MR{3575206}

\bibitem{Gorenstein1968}
Daniel Gorenstein, \emph{Finite groups}, Harper \& Row Publishers, New York,
  1968. \MR{MR0231903 (38 \#229)}

\bibitem{Jones2002}
Gareth~A. Jones, \emph{Cyclic regular subgroups of primitive permutation
  groups}, J. Group Theory \textbf{5} (2002), no.~4, 403--407. \MR{MR1931365
  (2003h:20004)}

\bibitem{Joseph1995}
Anne Joseph, \emph{The isomorphism problem for {C}ayley digraphs on groups of
  prime-squared order}, Discrete Math. \textbf{141} (1995), no.~1-3, 173--183.
  \MR{1336683 (96e:05071)}

\bibitem{KovacsS2012}
Istv{\'a}n Kov{\'a}cs and Mary Servatius, \emph{On {C}ayley digraphs on
  nonisomorphic 2-groups}, J. Graph Theory \textbf{70} (2012), no.~4, 435--448.
  \MR{2957057}

\bibitem{Li2003}
Cai~Heng Li, \emph{The finite primitive permutation groups containing an
  abelian regular subgroup}, Proc. London Math. Soc. (3) \textbf{87} (2003),
  no.~3, 725--747. \MR{MR2005881 (2004i:20003)}

\bibitem{LiebeckPS2010}
Martin~W. Liebeck, Cheryl~E. Praeger, and Jan Saxl, \emph{Regular subgroups of
  primitive permutation groups}, Mem. Amer. Math. Soc. \textbf{203} (2010),
  no.~952, vi+74. \MR{2588738}

\bibitem{MarusicM2005}
Dragan Maru{\v{s}}i{\v{c}} and Joy Morris, \emph{Normal circulant graphs with
  noncyclic regular subgroups}, J. Graph Theory \textbf{50} (2005), no.~1,
  13--24. \MR{2157535 (2006c:05073)}

\bibitem{Meldrum1995}
J.~D.~P. Meldrum, \emph{Wreath products of groups and semigroups}, Pitman
  Monographs and Surveys in Pure and Applied Mathematics, vol.~74, Longman,
  Harlow, 1995. \MR{MR1379113 (97j:20030)}

\bibitem{MMV}
Luke Morgan, Joy Morris, and Gabriel Verret, \emph{Digraphs with small
  automorphism groups that are {C}ayley on two nonisomorphic groups}, Art
  Discrete Appl. Math. \textbf{3} (2020), no.~1, \#P1.01, 11. \MR{4104211}

\bibitem{Morris1996}
Joy Morris, \emph{Isomorphic {C}ayley graphs on different groups}, Proceedings
  of the {T}wenty-seventh {S}outheastern {I}nternational {C}onference on
  {C}ombinatorics, {G}raph {T}heory and {C}omputing ({B}aton {R}ouge, {LA},
  1996), vol. 121, 1996, pp.~93--96. \MR{MR1431979 (97k:05102)}

\bibitem{Morris1999}
\bysame, \emph{Isomorphic {C}ayley graphs on nonisomorphic groups}, J. Graph
  Theory \textbf{31} (1999), no.~4, 345--362. \MR{MR1698752 (2000e:05085)}

\bibitem{Muzychuk2004}
M.~Muzychuk, \emph{A solution of the isomorphism problem for circulant graphs},
  Proc. London Math. Soc. (3) \textbf{88} (2004), no.~1, 1--41. \MR{MR2018956
  (2004h:05084)}

\bibitem{Muzychuk1999}
Mikhail Muzychuk, \emph{On the isomorphism problem for cyclic combinatorial
  objects}, Discrete Math. \textbf{197/198} (1999), 589--606, 16th British
  Combinatorial Conference (London, 1997). \MR{MR1674890 (2000e:05165)}

\bibitem{Palfy1987}
P.~P. P{\'a}lfy, \emph{Isomorphism problem for relational structures with a
  cyclic automorphism}, European J. Combin. \textbf{8} (1987), no.~1, 35--43.
  \MR{MR884062 (88i:05097)}

\bibitem{Wielandt1969}
H.~Wielandt, \emph{Permutation groups through invariant relations and invariant
  functions}, lectures given at The Ohio State University, Columbus, Ohio,
  1969.

\end{thebibliography}
\end{document}